\documentclass[11pt,a4paper]{amsart}
\usepackage{amsfonts,amsmath,amssymb,amsthm,mathtools}
\usepackage{times}
\usepackage{color}
\usepackage{graphicx}

\numberwithin{equation}{section}
\renewcommand{\figurename}{Fig.}

\DeclareSymbolFont{SY}{U}{psy}{m}{n}
\DeclareMathSymbol{\emptyset}{\mathord}{SY}{'306}

\DeclareMathOperator{\Ran}{Ran}
\DeclareMathOperator{\Dom}{Dom}
\DeclareMathOperator{\spec}{spec}
\DeclareMathOperator{\dist}{dist}

\DeclarePairedDelimiter{\abs}{\lvert}{\rvert}
\DeclarePairedDelimiter{\norm}{\lVert}{\rVert}

\newcommand{\R}{\mathbb{R}}
\newcommand{\N}{\mathbb{N}}
\newcommand{\EE}{\mathsf{E}}
\newcommand{\dd}{\mathrm{d}}

\newcommand{\cH}{{\mathcal H}}
\newcommand{\cO}{{\mathcal O}}

\newtheorem{introtheorem}{Theorem}{\bf}{\it}
\newtheorem{theorem}{Theorem}[section]{\bf}{\it}
\newtheorem{proposition}[theorem]{Proposition}{\bf}{\it}
\newtheorem{corollary}[theorem]{Corollary}{\bf}{\it}
\newtheorem{lemma}[theorem]{Lemma}{\bf}{\it}
\newtheorem{remark}[theorem]{Remark}{\it}{\rm}
\newtheorem{definition}[theorem]{Definition}{\bf}{\it}
\newtheorem{hypothesis}[theorem]{Hypothesis}{\bf}{\it}

\title[On an estimate in the subspace perturbation problem]{On an estimate in the subspace perturbation problem}
\date{}

\author[A.\ Seelmann]{Albrecht Seelmann$^*$}
\address{A.~Seelmann, FB 08 - Institut f\"{u}r Mathematik,
Johannes Gutenberg-Universi\-t\"{a}t Mainz,
Staudinger Weg 9,
D-55099 Mainz,
Germany}
\email{seelmann@mathematik.uni-mainz.de}
\thanks{$^*$The material presented in this work will be part of the author's Ph.D. thesis.}

\begin{document}

%%%%%%%%%%%%%%%%%%%%%%%%%%%%%%%%%%%%%%%%%%%%%%%%%%%%%%%%%%%%%%%%%%%%%%%%%%%%%%%%%%%%%%%%%%%%%%%%%%%%%%%%%%%%%%%%%%%%%%%%%%%%%%%%%%%
%%% Abstract
%%%%%%%%%%%%%%%%%%%%%%%%%%%%%%%%%%%%%%%%%%%%%%%%%%%%%%%%%%%%%%%%%%%%%%%%%%%%%%%%%%%%%%%%%%%%%%%%%%%%%%%%%%%%%%%%%%%%%%%%%%%%%%%%%%%
\begin{abstract}
The problem of variation of spectral subspaces for linear self-adjoint operators under an additive bounded perturbation is
considered. The aim is to find the best possible upper bound on the norm of the difference of two spectral projections associated
with isolated parts of the spectrum of the perturbed and unperturbed operators.

In the approach presented here, a constrained optimization problem on a specific set of parameters is formulated, whose solution
yields an estimate on the arcsine of the norm of the difference of the corresponding spectral projections. The problem is solved
explicitly. This optimizes the approach by Albeverio and Motovilov in [Complex Anal.\ Oper.\ Theory \textbf{7} (2013), 1389--1416].
In particular, the resulting estimate is stronger than the one obtained there.
\end{abstract}

\maketitle

%%%%%%%%%%%%%%%%%%%%%%%%%%%%%%%%%%%%%%%%%%%%%%%%%%%%%%%%%%%%%%%%%%%%%%%%%%%%%%%%%%%%%%%%%%%%%%%%%%%%%%%%%%%%%%%%%%%%%%%%%%%%%%%%%%%
%%%%%%%%%%%%%%%%%%%%%%%%%%%%%%%%%%%%%%%%%%%%%%%%%%%%%%%%%%%%%%%%%%%%%%%%%%%%%%%%%%%%%%%%%%%%%%%%%%%%%%%%%%%%%%%%%%%%%%%%%%%%%%%%%%%
%%%% Introduction and main result
%%%%%%%%%%%%%%%%%%%%%%%%%%%%%%%%%%%%%%%%%%%%%%%%%%%%%%%%%%%%%%%%%%%%%%%%%%%%%%%%%%%%%%%%%%%%%%%%%%%%%%%%%%%%%%%%%%%%%%%%%%%%%%%%%%%
%%%%%%%%%%%%%%%%%%%%%%%%%%%%%%%%%%%%%%%%%%%%%%%%%%%%%%%%%%%%%%%%%%%%%%%%%%%%%%%%%%%%%%%%%%%%%%%%%%%%%%%%%%%%%%%%%%%%%%%%%%%%%%%%%%%
\section{Introduction and the main result}\label{sec:intro}
Let $A$ be a self-adjoint possibly unbounded operator on a separable Hilbert space $\cH$ such that the spectrum of $A$ is separated
into two disjoint components, that is,
\begin{equation}\label{eq:specSep}
 \spec(A)=\sigma\cup\Sigma\quad\text{ with }\quad d:=\dist(\sigma,\Sigma)>0\,.
\end{equation}
Let $V$ be a bounded self-adjoint operator on $\cH$.

It is well known (see, e.g., \cite[Theorem V.4.10]{Kato66}) that the spectrum of the perturbed self-adjoint operator $A+V$ is
confined in the closed $\norm{V}$-neighbourhood of the spectrum of the unperturbed operator $A$, that is,
\begin{equation}\label{eq:specPert}
 \spec(A+V)\subset \overline{\cO_{\norm{V}}\bigl(\spec(A)\bigr)}\,,
\end{equation}
where $\cO_{\norm{V}}\bigl(\spec(A)\bigr)$ denotes the open $\norm{V}$-neighbourhood of $\spec(A)$. In particular, if
\begin{equation}\label{eq:pertNormBound}
 \norm{V} < \frac{d}{2}\,,
\end{equation}
then the spectrum of the operator $A+V$ is likewise separated into two disjoint components $\omega$ and $\Omega$, where
\[
 \omega=\spec(A+V)\cap \cO_{d/2}(\sigma)\quad \text{ and }\quad \Omega=\spec(A+V)\cap \cO_{d/2}(\Sigma)\,.
\]
Therefore, under condition \eqref{eq:pertNormBound}, the two components of the spectrum of $A+V$ can be interpreted as
perturbations of the corresponding original spectral components $\sigma$ and $\Sigma$ of $\spec(A)$. Clearly, the condition
\eqref{eq:pertNormBound} is sharp in the sense that if $\norm{V}\ge d/2$, the spectrum of the perturbed operator $A+V$ may not have
separated components at all.

The effect of the additive perturbation $V$ on the spectral subspaces for $A$ is studied in terms of the corresponding spectral
projections. Let $\EE_A(\sigma)$ and $\EE_{A+V}\bigl(\cO_{d/2}(\sigma)\bigr)$ denote the spectral projections for $A$ and $A+V$
associated with the Borel sets $\sigma$ and $\cO_{d/2}(\sigma)$, respectively. It is well known that
$\norm{\EE_A(\sigma)-\EE_{A+V}\bigl(\cO_{d/2}(\sigma)\bigr)}\le 1$ since the corresponding inequality holds for every difference of
orthogonal projections in $\cH$, see, e.g., \cite[Section 34]{AG93}. Moreover, if
\begin{equation}\label{eq:projAcute}
 \norm{\EE_A(\sigma)-\EE_{A+V}\bigl(\cO_{d/2}(\sigma)\bigr)} < 1\,,
\end{equation}
then the spectral projections $\EE_A(\sigma)$ and $\EE_{A+V}\bigl(\cO_{d/2}(\sigma)\bigr)$ are unitarily equivalent, see, e.g.,
\cite[Theorem I.6.32]{Kato66}.

In this sense, if inequality \eqref{eq:projAcute} holds, the spectral subspace $\Ran\EE_{A+V}\bigl(\cO_{d/2}(\sigma)\bigr)$ can be
understood as a rotation of the unperturbed spectral subspace $\Ran\EE_A(\sigma)$. The quantity
\[
 \arcsin\bigl(\norm{\EE_A(\sigma) - \EE_{A+V}\bigl(\cO_{d/2}(\sigma)\bigr)}\bigr)
\]
serves as a measure for this rotation and is called the \emph{maximal angle} between the spectral subspaces $\Ran\EE_A(\sigma)$ and
$\Ran\EE_{A+V}\bigl(\cO_{d/2}(\sigma)\bigr)$. A short survey on the concept of the maximal angle between closed subspaces of a
Hilbert space can be found in \cite[Section 2]{AM13}; see also \cite{DK70}, \cite[Theorem 2.2]{KMM03:2}, \cite[Section 2]{Seel13},
and references therein.

It is a natural question whether the bound \eqref{eq:pertNormBound} is sufficient for inequality \eqref{eq:projAcute} to hold, or
if one has to impose a stronger bound on the norm of the perturbation $V$ in order to ensure \eqref{eq:projAcute}. Basically, the
following two problems arise:
\begin{enumerate}
 \renewcommand{\theenumi}{\roman{enumi}}
 \item What is the best possible constant $c_\text{opt}\in\bigl(0,\frac{1}{2}\bigr]$ such that
       \[
        \arcsin\bigl(\norm{\EE_A(\sigma)-\EE_{A+V}(\cO_{d/2}\bigl(\sigma)\bigr)}\bigr)<\frac{\pi}{2}\quad\text{ whenever }\quad
        \norm{V}<c_\text{opt}\cdot d\ ?
       \]
 \item Which function $f\colon[0,c_\text{opt})\to\bigl[0,\frac{\pi}{2}\bigr)$ is best possible in the estimate
       \[
        \arcsin\bigl(\norm{\EE_A(\sigma)-\EE_{A+V}(\cO_{d/2}\bigl(\sigma)\bigr)}\bigr) \le f\biggl(\frac{\norm{V}}{d}\biggr)\,,
        \quad \norm{V}<c_\text{opt}\cdot d\ ?
       \]
\end{enumerate}
Both the constant $c_{\mathrm{opt}}$ and the function $f$ are supposed to be universal in the sense that they are independent of
the operators $A$ and $V$.

Note that we have made no assumptions on the disposition of the spectral components $\sigma$ and $\Sigma$ other than
\eqref{eq:specSep}. If, for example, $\sigma$ and $\Sigma$ are additionally assumed to be subordinated, that is,
$\sup\sigma<\inf\Sigma$ or vice versa, or if one of the two sets lies in a finite gap of the other one, then the corresponding best
possible constant in problem (i) is known to be $\frac{1}{2}$, and the best possible function $f$ in problem (ii) is given by
$f(x)=\frac{1}{2}\arcsin\bigl(2x\bigr)$, see, e.g., \cite[Lemma 2.3]{KMM03} and \cite[Theorem 5.1]{Davis63}; see also
\cite[Remark 2.9]{Seel13}.

However, under the sole assumption \eqref{eq:specSep}, both problems are still unsolved. It has been conjectured that
$c_\text{opt}=\frac{1}{2}$ (see \cite{AM13}; cf.\ also \cite{KMM03} and \cite{KMM07}), but there is no proof available for that
yet. So far, only lower bounds on the optimal constant $c_\text{opt}$ and upper bounds on the best possible function $f$ can be
given. For example, in \cite[Theorem 1]{KMM03} it was shown that
\[
 c_{\mathrm{opt}} \ge \frac{2}{2+\pi}=0.3889845\ldots
\]
and
\begin{equation}\label{eq:KMM}
 f(x) \le \arcsin\Bigl(\frac{\pi}{2}\,\frac{x}{1-x}\Bigr)<\frac{\pi}{2}\quad\text{ for }\quad 0\le x < \frac{2}{2+\pi}\,.
\end{equation}
In \cite[Theorem 6.1]{MS10} this result was strengthened to
\[
 c_{\mathrm{opt}} \ge \frac{\sinh(1)}{\exp(1)}=0{.}4323323\ldots
\]
and
\begin{equation}\label{eq:MS}
 f(x) \le \frac{\pi}{4}\log\Bigl(\frac{1}{1-2x}\Bigr) < \frac{\pi}{2}\quad\text{ for }\quad 0\le x  <\frac{\sinh(1)}{\exp(1)}\,.
\end{equation}
Recently, Albeverio and Motovilov have shown in \cite[Theorem 5.4]{AM13} that
\begin{equation}\label{eq:AMConst}
 c_\text{opt} \ge c_* = 16\,\frac{\pi^6-2\pi^4+32\pi^2-32}{(\pi^2+4)^4} = 0{.}4541692\ldots
\end{equation}
and
\begin{equation}\label{eq:AM}
 f(x) \le M_*(x)<\frac{\pi}{2}\quad \text{ for }\quad 0\le x < c_*\,,
\end{equation}
where
\begin{equation}\label{eq:AMFunc}
 M_*(x)=
 \begin{cases}
  \frac{1}{2}\arcsin(\pi x) & \text{for}\quad 0\le x\le \frac{4}{\pi^2+4}\,,\\[0.1cm]
  \frac{1}{2}\arcsin\bigl(\frac{4\pi}{\pi^2+4}\bigr) + \frac{1}{2}\arcsin\Bigl(\pi\,\frac{(\pi^2+4)x-4}
    {\pi^2-4}\Bigr) & \text{for}\quad \frac{4}{\pi^2+4} < x \le \frac{8\pi^2}{(\pi^2+4)^2}\,,\\[0.2cm]
  \arcsin\bigl(\frac{4\pi}{\pi^2+4}\bigr) + \frac{1}{2}\arcsin\Bigl(\pi\,\frac{(\pi^2+4)^2x-8\pi^2}
    {(\pi^2-4)^2}\Bigr) & \text{for}\quad \frac{8\pi^2}{(\pi^2+4)^2} < x \le c_*\,.
 \end{cases}
\end{equation}
It should be noted that the first two results \eqref{eq:KMM} and \eqref{eq:MS} were originally formulated in \cite{KMM03} and
\cite{MS10}, respectively, only for the case where the operator $A$ is assumed to be bounded. However, both results admit an
immediate, straightforward generalization to the case where the operator $A$ is allowed to be unbounded, see, e.g.,
\cite[Proposition 3.4 and Theorem 3.5]{AM13}.

The aim of the present work is to sharpen the estimate \eqref{eq:AM}. More precisely, our main result is as follows.

\begin{introtheorem}\label{thm:mainResult}
 Let $A$ be a self-adjoint operator on a separable Hilbert space $\cH$ such that the spectrum of $A$ is separated into two disjoint
 components, that is,
 \[
  \spec(A)=\sigma\cup\Sigma\quad\text{ with }\quad d:=\dist(\sigma,\Sigma)>0\,.
 \]
 Let $V$ be a bounded self-adjoint operator on $\cH$ satisfying
 \[
  \norm{V} < c_\mathrm{crit}\cdot d
 \]
 with
 \[
  c_\mathrm{crit}=\frac{1-\bigl(1-\frac{\sqrt{3}}{\pi}\bigr)^3}{2}=3\sqrt{3}\,\frac{\pi^2-\sqrt{3}\pi+1}{2\pi^3}=
  0{.}4548399\ldots
 \]
 Then, the spectral projections $\EE_A(\sigma)$ and $\EE_{A+V}\bigl(\cO_{d/2}(\sigma)\bigr)$ for the self-adjoint operators $A$ and
 $A+V$ associated with $\sigma$ and the open $\frac{d}{2}$-neighbourhood $\cO_{d/2}(\sigma)$ of $\sigma$, respectively, satisfy the
 estimate
 \begin{equation}\label{eq:mainResult}
  \arcsin\bigl(\norm{\EE_A({\sigma})-\EE_{A+V}\bigl(\cO_{d/2}(\sigma)\bigr)}\bigr) \le N\biggl(\frac{\norm{V}}{d}\biggr)
  < \frac{\pi}{2}\,,
 \end{equation}
 where the function $N\colon[0,c_\mathrm{crit}]\to\bigl[0,\frac{\pi}{2}\bigr]$ is given by
 \begin{equation}\label{eq:mainResultFunc}
  N(x) =
  \begin{cases}
   \frac{1}{2}\arcsin(\pi x) & \text{ for }\quad 0\le x\le \frac{4}{\pi^2+4}\,,\\[0.15cm]
   \arcsin\Bigl(\sqrt{\frac{2\pi^2x-4}{\pi^2-4}}\,\Bigr) & \text{ for }\quad \frac{4}{\pi^2+4} < x < 4\,\frac{\pi^2-2}
     {\pi^4}\,,\\[0.15cm]
   \arcsin\bigl(\frac{\pi}{2}(1-\sqrt{1-2x}\,)\bigr) & \text{ for }\quad 4\,\frac{\pi^2-2}{\pi^4} \le x \le \kappa\,,\\[0.15cm]
   \frac{3}{2}\arcsin\bigl(\frac{\pi}{2}(1-\sqrt[\leftroot{4}3]{1-2x}\,)\bigr) & \text{ for }\quad
     \kappa < x \le c_\mathrm{crit}\,.
  \end{cases}
 \end{equation}
 Here, $\kappa\in\bigl(4\frac{\pi^2-2}{\pi^4},2\frac{\pi-1}{\pi^2}\bigr)$ is the unique solution to the equation
 \begin{equation}\label{eq:kappa}
  \arcsin\Bigl(\frac{\pi}{2}\bigl(1-\sqrt{1-2\kappa}\,\bigr)\Bigr)=\frac{3}{2}\arcsin\Bigl(\frac{\pi}{2}
  \bigl(1-\sqrt[\leftroot{4}3]{1-2\kappa}\,\bigr)\Bigr)
 \end{equation}
 in the interval $\bigl(0,2\frac{\pi-1}{\pi^2}\bigr]$. The function $N$ is strictly increasing, continuous on
 $[0,c_\mathrm{crit}]$, and continuously differentiable on $(0,c_\mathrm{crit})\setminus\{\kappa\}$.
\end{introtheorem}
Numerical calculations give $\kappa=0{.}4098623\ldots$

The estimate \eqref{eq:mainResult} in Theorem \ref{thm:mainResult} remains valid if the constant $\kappa$ in the definition of the
function $N$ is replaced by any other constant within the interval $\bigl(4\frac{\pi^2-2}{\pi^4},2\frac{\pi-1}{\pi^2}\bigr)$, see
Remark \ref{rem:kappaRepl} below. However, the particular choice \eqref{eq:kappa} ensures that the function $N$ is continuous and
as small as possible. In particular, we have $N(x) = M_*(x)$ for $0\le x\le\frac{4}{\pi^2+4}$ and
\[
  N(x) < M_*(x)\quad\text{ for }\quad \frac{4}{\pi^2+4} < x \le c_*\,,
\]
where $c_*$ and $M_*$ are given by \eqref{eq:AMConst} and \eqref{eq:AMFunc} respectively, see Remark \ref{rem:estOptimality} below.

From Theorem \ref{thm:mainResult} we immediately deduce that
\[
 c_{\mathrm{opt}}\ge c_\mathrm{crit} > c_*
\]
and
\[
 f(x) \le N(x)<\frac{\pi}{2}\quad\text{ for }\quad 0\le x< c_\mathrm{crit}\,.
\]
Both are the best respective bounds for the two problems (i) and (ii) known so far.

The paper is organized as follows:
In Section \ref{sec:optProb}, based on the triangle inequality for the maximal angle and a suitable a priori rotation bound for
small perturbations (see Proposition \ref{prop:genRotBound}), we formulate a constrained optimization problem, whose solution
provides an estimating function for the maximal angle between the corresponding spectral subspaces, see Definition
\ref{def:optProb}, Proposition \ref{prop:mainEstimate}, and Theorem \ref{thm:solOptProb}. In this way, the approach by Albeverio
and Motovilov in \cite{AM13} is optimized and, in particular, a proof of Theorem \ref{thm:mainResult} is obtained. The explicit
solution to the optimization problem is given in Theorem \ref{thm:solOptProb}, which is proved in Section \ref{sec:solOptProb}. The
technique used there involves variational methods and may also be useful for solving optimization problems of a similar structure.

Finally, Appendix \ref{app:sec:inequalities} is devoted to some elementary inequalities used in Section \ref{sec:solOptProb}.

%%%%%%%%%%%%%%%%%%%%%%%%%%%%%%%%%%%%%%%%%%%%%%%%%%%%%%%%%%%%%%%%%%%%%%%%%%%%%%%%%%%%%%%%%%%%%%%%%%%%%%%%%%%%%%%%%%%%%%%%%%%%%%%%%%%
%%%%%%%%%%%%%%%%%%%%%%%%%%%%%%%%%%%%%%%%%%%%%%%%%%%%%%%%%%%%%%%%%%%%%%%%%%%%%%%%%%%%%%%%%%%%%%%%%%%%%%%%%%%%%%%%%%%%%%%%%%%%%%%%%%%
%%%% Formulation of the optimization problem and proof of the main result
%%%%%%%%%%%%%%%%%%%%%%%%%%%%%%%%%%%%%%%%%%%%%%%%%%%%%%%%%%%%%%%%%%%%%%%%%%%%%%%%%%%%%%%%%%%%%%%%%%%%%%%%%%%%%%%%%%%%%%%%%%%%%%%%%%%
%%%%%%%%%%%%%%%%%%%%%%%%%%%%%%%%%%%%%%%%%%%%%%%%%%%%%%%%%%%%%%%%%%%%%%%%%%%%%%%%%%%%%%%%%%%%%%%%%%%%%%%%%%%%%%%%%%%%%%%%%%%%%%%%%%%
\section{An optimization problem}\label{sec:optProb}
In this section, we formulate a constrained optimization problem, whose solution provides an estimate on the maximal angle between
the spectral subspaces associated with isolated parts of the spectrum of the corresponding perturbed and unperturbed operators,
respectively. In particular, this yields a proof of Theorem \ref{thm:mainResult}.

We make the following notational setup.

\begin{hypothesis}\label{app:hypHyp}
 Let $A$ be as in Theorem \ref{thm:mainResult}, and let $V\neq0$ be a bounded self-adjoint operator on the Hilbert space $\cH$. For
 $0\le t<\frac{1}{2}$, introduce $B_t:=A+td\,\frac{V}{\norm{V}}$, $\Dom(B_t):=\Dom(A)$, and denote by
 $P_t:=\EE_{B_t}\bigl(\cO_{d/2}(\sigma)\bigr)$ the spectral projection for $B_t$ associated with the open
 $\frac{d}{2}$-neighbourhood $\cO_{d/2}(\sigma)$ of $\sigma$.
\end{hypothesis}

Under Hypothesis \ref{app:hypHyp}, one has $\norm{B_t-A}=td<\frac{d}{2}$ for $0\le t<\frac{1}{2}$. Taking into account the
inclusion \eqref{eq:specPert}, the spectrum of each $B_t$ is likewise separated into two disjoint components, that is,
\[
 \spec(B_t) = \omega_t\cup\Omega_t\quad\text{ for }\quad 0\le t<\frac{1}{2}\,,
\]
where
\[
 \omega_t=\spec(B_t)\cap\overline{\cO_{td}(\sigma)}\quad\text{ and }\quad
 \Omega_t=\spec(B_t)\cap\overline{\cO_{td}(\Sigma)}\,.
\]
In particular, one has
\begin{equation}\label{eq:specGapBound}
 \delta_t:=\dist(\omega_t,\Omega_t) \ge (1-2t)d>0 \quad\text{ for }\quad 0\le t<\frac{1}{2}\,.
\end{equation}
Moreover, the mapping $\bigl[0,\frac{1}{2}\bigr)\ni t\mapsto P_t$ is norm continuous, see, e.g., \cite[Theorem 3.5]{AM13};
cf.\ also the forthcoming estimate \eqref{eq:localKMM}.

For arbitrary $0\le r\le s<\frac{1}{2}$, we can consider $B_s=B_r+(s-r)d\frac{V}{\norm{V}}$ as a perturbation of $B_r$. Taking into
account the a priori bound \eqref{eq:specGapBound}, we then observe that
\begin{equation}\label{eq:localPert}
 \frac{\norm{B_s-B_r}}{\delta_r} = \frac{(s-r)d}{\dist(\omega_r,\Omega_r)}\le\frac{s-r}{1-2r} < \frac{1}{2} \quad\text{ for }\quad
 0\le r\le s<\frac{1}{2}\,.
\end{equation}
Furthermore, it follows from \eqref{eq:localPert} and the inclusion \eqref{eq:specPert} that $\omega_s$ is exactly the part of
$\spec(B_s)$ that is contained in the open $\frac{\delta_r}{2}$-neighbourhood of $\omega_r$, that is,
\begin{equation}\label{eq:localSpecPert}
 \omega_s = \spec(B_s)\cap\cO_{\delta_r/2}(\omega_r) \quad\text{ for }\quad 0\le r\le s<\frac{1}{2}\,.
\end{equation}

Let $t\in\bigl(0,\frac{1}{2}\bigr)$ be arbitrary, and let $0=t_0<t_1<\dots<t_{n+1}=t$ with $n\in\N_0$ be a finite partition of the
interval $[0,t]$. Define
\begin{equation}\label{eq:defLambda}
 \lambda_j := \frac{t_{j+1}-t_j}{1-2t_j}<\frac{1}{2}\,,\quad j=0,\dots,n\,.
\end{equation}
Recall that the mapping $\rho$ given by
\begin{equation}\label{eq:metric}
 \rho(P,Q)=\arcsin\bigl(\norm{P-Q}\bigr)\quad\text{ with }\quad P,Q\ \text{ orthogonal projections in }\cH\,,
\end{equation}
defines a metric on the set of orthogonal projections in $\cH$, see \cite{Brown93}, and also \cite[Lemma 2.15]{AM13} and
\cite{MS10}. Using the triangle inequality for this metric, we obtain
\begin{equation}\label{eq:projTriangle}
 \arcsin\bigl(\norm{P_0-P_t}\bigr) \le \sum_{j=0}^n \arcsin\bigl(\norm{P_{t_j}-P_{t_{j+1}}}\bigr)\,.
\end{equation}
Considering $B_{t_{j+1}}$ as a perturbation of $B_{t_j}$, it is clear from \eqref{eq:localPert} and \eqref{eq:localSpecPert} that
each summand of the right-hand side of \eqref{eq:projTriangle} can be treated in the same way as the maximal angle in the general
situation discussed in Section \ref{sec:intro}. For example, combining \eqref{eq:localPert}--\eqref{eq:defLambda} with the bound
\eqref{eq:KMM} yields
\begin{equation}\label{eq:localKMM}
 \norm{P_{t_j}-P_{t_{j+1}}} \le \frac{\pi}{2}\,\frac{\lambda_j}{1-\lambda_j} = \frac{\pi}{2}\,\frac{t_{j+1}-t_j}{1-t_j-t_{j+1}} \le
 \frac{\pi}{2}\,\frac{t_{j+1}-t_j}{1-2t_{j+1}}\,,\quad j=0,\dots,n\,,
\end{equation}
where we have taken into account that $\norm{P_{t_j}-P_{t_{j+1}}}\le1$ and that $\frac{\pi}{2}\frac{\lambda_j}{1-\lambda_j}\ge 1$
if $\lambda_j\ge\frac{2}{2+\pi}$.

Obviously, the estimates \eqref{eq:projTriangle} and \eqref{eq:localKMM} hold for arbitrary finite partitions of the interval
$[0,t]$. In particular, if partitions with arbitrarily small mesh size are considered, then, as a result of
$\frac{t_{j+1}-t_j}{1-2t_{j+1}}\le\frac{t_{j+1}-t_j}{1-2t}$, the norm of each corresponding projector difference in
\eqref{eq:localKMM} is arbitrarily small as well. At the same time, the corresponding Riemann sums
\[
 \sum_{j=0}^n \frac{t_{j+1}-t_j}{1-2t_{j+1}}
\]
are arbitrarily close to the integral $\int_0^t\frac{1}{1-2\tau}\,\dd\tau$. Since $\frac{\arcsin(x)}{x}\to1$ as $x\to0$, we
conclude from \eqref{eq:projTriangle} and \eqref{eq:localKMM} that
\[
 \arcsin\bigl(\norm{P_0-P_t}\bigr) \le \frac{\pi}{2}\int_0^t \frac{1}{1-2\tau}\,\dd\tau=
 \frac{\pi}{4}\log\Bigl(\frac{1}{1-2t}\Bigr)\,.
\]
Once the bound \eqref{eq:KMM} has been generalized to the case where the operator $A$ is allowed to be unbounded, this argument is
an easy and straightforward way to prove the bound \eqref{eq:MS}.

Albeverio and Motovilov demonstrated in \cite{AM13} that a stronger result can be obtained from \eqref{eq:projTriangle}. They
considered a specific finite partition of the interval $[0,t]$ and used a suitable a priori bound (see
\cite[Corollary 4.3 and Remark 4.4]{AM13}) to estimate the corresponding summands of the right-hand side of
\eqref{eq:projTriangle}. This a priori bound, which is related to the Davis-Kahan $\sin2\Theta$ theorem from \cite{DK70}, is used
in the present work as well. We therefore state the corresponding result in the following proposition for future reference. It
should be noted that our formulation of the statement slightly differs from the original one in \cite{AM13}. A justification of
this modification, as well as a deeper discussion on the material including an alternative, straightforward proof of the original
result \cite[Corollary 4.3]{AM13}, can be found in \cite{Seel13}.

\begin{proposition}[{\cite[Corollary 2]{Seel13}}]\label{prop:genRotBound}
 Let $A$ and $V$ be as in Theorem \ref{thm:mainResult}. If $\norm{V}\le \frac{d}{\pi}$, then the spectral projections
 $\EE_{A}(\sigma)$ and $\EE_{A+V}\bigl(\cO_{d/2}(\sigma)\bigr)$ for the self-adjoint operators $A$ and $A+V$ associated with the
 Borel sets $\sigma$ and $\cO_{d/2}(\sigma)$, respectively, satisfy the estimate
 \[
  \arcsin\bigl(\norm{\EE_A(\sigma)-\EE_{A+V}\bigl(\cO_{d/2}(\sigma)\bigr)}\bigr) \le
  \frac{1}{2}\arcsin\Bigl(\frac{\pi}{2}\cdot 2\,\frac{\norm{V}}{d}\Bigr)\le\frac{\pi}{4}\,.
 \] 
\end{proposition}

The estimate given by Proposition \ref{prop:genRotBound} is universal in the sense that the estimating function
$x\mapsto\frac{1}{2}\arcsin(\pi x)$ depends neither on the unperturbed operator $A$ nor on the perturbation $V$. Moreover, for
perturbations $V$ satisfying $\norm{V}\le\frac{4}{\pi^2+4}\,d$, this a priori bound on the maximal angle between the corresponding
spectral subspaces is the strongest one available so far, cf.\ \cite[Remark 5.5]{AM13}.

Assume that the given partition of the interval $[0,t]$ additionally satisfies
\begin{equation}\label{eq:seqParamCond}
 \lambda_j = \frac{t_{j+1}-t_j}{1-2t_j}\le\frac{1}{\pi}\,,\quad j=0,\dots,n\,.
\end{equation}
In this case, it follows from \eqref{eq:localPert}, \eqref{eq:localSpecPert}, \eqref{eq:projTriangle}, and Proposition
\ref{prop:genRotBound} that
\begin{equation}\label{eq:essEstimate}
 \arcsin\bigl(\norm{P_0-P_t}\bigr) \le \frac{1}{2}\sum_{j=0}^n \arcsin(\pi\lambda_j)\,.
\end{equation}

Along with a specific choice of the partition of the interval $[0,t]$, estimate \eqref{eq:essEstimate} is the essence of the
approach by Albeverio and Motovilov in \cite{AM13}. In the present work, we optimize the choice of the partition of the interval
$[0,t]$, so that for every fixed parameter $t$ the right-hand side of inequality \eqref{eq:essEstimate} is minimized. An equivalent
and more convenient reformulation of this approach is to maximize the parameter $t$ in estimate \eqref{eq:essEstimate} over all
possible choices of the parameters $n$ and $\lambda_j$ for which the right-hand side of \eqref{eq:essEstimate} takes a fixed value.

Obviously, we can generalize estimate \eqref{eq:essEstimate} to the case where the finite sequence $(t_j)_{j=1}^n$ is allowed to be
just increasing and not necessarily strictly increasing. Altogether, this motivates the following considerations.

\begin{definition}\label{def:params}
 For $n\in\N_0$ define
 \[
  D_n:=\Bigl\{ (\lambda_j)\in l^1(\N_0) \Bigm| 0\le \lambda_j\le\frac{1}{\pi}\ \
  \text{ for }\ \ j\le n\ \ \text{ and }\ \ \lambda_j=0\ \ \text{ for }\ \ j\ge n+1 \Bigr\}\,,
 \]
 and let $D:=\bigcup_{n\in\N_0} D_n$.
\end{definition}

Every finite partition of the interval $[0,t]$ that satisfies condition \eqref{eq:seqParamCond} is related to a sequence in $D$ in
the obvious way. Conversely, the following lemma allows to regain the finite partition of the interval $[0,t]$ from this sequence.

\begin{lemma}\label{lem:seq}
 \hspace*{2cm}

 \begin{enumerate}
  \renewcommand{\theenumi}{\alph{enumi}}
  \item For every $x\in\bigl[0,\frac{1}{2}\bigr)$ the mapping $\bigl[0,\frac{1}{2}\bigr]\ni t\mapsto t+x(1-2t)$ is strictly
        increasing.
  \item For every $\lambda=(\lambda_j)\in D$ the sequence $(t_j)\subset\R$ given by the recursion
        \begin{equation}\label{eq:seqDef}
         t_{j+1} = t_j + \lambda_j (1-2t_j)\,,\quad j\in\N_0\,,\quad t_0=0\,,
        \end{equation}
        is increasing and satisfies $0\le t_j<\frac{1}{2}$ for all $j\in\N_0$. Moreover, one has $t_j=t_{n+1}$ for $j\ge n+1$ if
        $\lambda\in D_n$. In particular, $(t_j)$ is eventually constant.
 \end{enumerate}
 
 \begin{proof}
  The proof of claim (a) is straightforward and is hence omitted.

  For the proof of (b), let $\lambda=(\lambda_j)\in D$ be arbitrary and let $(t_j)\subset\R$ be given by \eqref{eq:seqDef}. Observe
  that $t_0=0<\frac{1}{2}$ and that (a) implies that
  \[
   0 \le t_{j+1}=t_j + \lambda_j(1-2t_j) < \frac{1}{2} + \lambda_j\Bigl(1-2\cdot\frac{1}{2}\Bigr) = \frac{1}{2}
   \quad\text{ if }\quad 0\le t_j<\frac{1}{2}\,.
  \]
  Thus, the two-sided estimate $0\le t_j<\frac{1}{2}$ holds for all $j\in\N_0$ by induction. In particular, it follows that
  $t_{j+1}-t_j=\lambda_j(1-2t_j)\ge 0$ for all $j\in\N_0$, that is, the sequence $(t_j)$ is increasing. Let $n\in\N_0$ such that
  $\lambda\in D_n$. Since $\lambda_j=0$ for $j\ge n+1$, it follows from the definition of $(t_j)$ that $t_{j+1}=t_j$ for
  $j\ge n+1$, that is, $t_j=t_{n+1}$ for $j\ge n+1$.
 \end{proof}%
\end{lemma}

It follows from part (b) of the preceding lemma that for every $\lambda\in D$ the sequence $(t_j)$ given by \eqref{eq:seqDef}
yields a finite partition of the interval $[0,t]$ with $t=\max_{j\in\N_0}t_j<\frac{1}{2}$. In this respect, the approach to
optimize the parameter $t$ in \eqref{eq:essEstimate} with a fixed right-hand side can now be formalized in the following way.

\begin{definition}\label{def:optProb}
 Let $W\colon D\to l^\infty(\N_0)$ denote the (non-linear) operator that maps every sequence in $D$ to the corresponding increasing
 and eventually constant sequence given by the recursion \eqref{eq:seqDef}. Moreover, let
 $M\colon\bigl[0,\frac{1}{\pi}\bigr]\to\bigl[0,\frac{\pi}{4}\bigr]$ be given by
 \[
  M(x):=\frac{1}{2}\arcsin(\pi x)\,.
 \]
 Finally, for $\theta\in\bigl[0,\frac{\pi}{2}\bigr]$ define
 \[
  D(\theta):=\biggl\{(\lambda_j)\in D \biggm| \sum_{j=0}^\infty M(\lambda_j)=\theta\biggr\}\subset D
 \]
 and
 \begin{equation}\label{eq:optProb}
  T(\theta):=\sup\bigl\{\max W(\lambda) \bigm| \lambda\in D(\theta)\bigr\}\,,
 \end{equation}
 where $\max W(\lambda):=\max_{j\in\N_0}t_j$ with $(t_j)=W(\lambda)$.
\end{definition}

For every fixed $\theta\in\bigl[0,\frac{\pi}{2}\bigr]$, it is easy to verify that indeed $D(\theta)\neq\emptyset$. Moreover, one
has $0\le T(\theta)\le\frac{1}{2}$ by part (b) of Lemma \ref{lem:seq}, and $T(\theta)=0$ holds if and only if $\theta=0$. In order
to compute $T(\theta)$ for $\theta\in\bigl(0,\frac{\pi}{2}\bigr]$, we have to maximize $\max W(\lambda)$ over
$\lambda\in D(\theta)\subset D$. This constrained optimization problem plays the central role in the approach presented in this
work.

The following proposition shows how this optimization problem is related to the problem of estimating the maximal angle between the
corresponding spectral subspaces.

\begin{proposition}\label{prop:mainEstimate}
 Assume Hypothesis \ref{app:hypHyp}. Let
 $\bigl[0,\frac{\pi}{2}\bigr]\ni\theta\mapsto S(\theta)\in\bigl[0,S\bigl(\frac{\pi}{2}\bigr)\bigr]\subset\bigl[0,\frac{1}{2}\bigr]$
 be a continuous, strictly increasing (hence invertible) mapping with
 \[
  0\le S(\theta) \le T(\theta)\quad\text{ for }\quad 0\le \theta<\frac{\pi}{2}\,.
 \]
 Then
 \[
  \arcsin\bigl(\norm{P_0-P_t}\bigr) \le S^{-1}(t)\quad\text{ for }\quad 0\le t< S\Bigl(\frac{\pi}{2}\Bigr)\,.
 \]
 
 \begin{proof}
  Since the mapping $\theta\mapsto S(\theta)$ is invertible, it suffices to show the inequality
  \begin{equation}\label{eq:mainEstForS}
   \arcsin\bigl(\norm{P_0-P_{S(\theta)}}\bigr) \le \theta\quad\text{ for }\quad 0\le\theta<\frac{\pi}{2}\,.
  \end{equation}

  Considering $T(0)=S(0)=0$, the case $\theta=0$ in inequality \eqref{eq:mainEstForS} is obvious. Let
  $\theta\in\bigl(0,\frac{\pi}{2}\bigr)$. In particular, one has $T(\theta)>0$. For arbitrary $t$ with $0\le t<T(\theta)$ choose
  $\lambda=(\lambda_j)\in D(\theta)$ such that $t<\max W(\lambda)\le T(\theta)$. Denote $(t_j):=W(\lambda)$. Since
  $t_j<\frac{1}{2}$ for all $j\in\N_0$ by part (b) of Lemma \ref{lem:seq}, it follows from the definition of $(t_j)$ that
  \begin{equation}\label{eq:locSeqI}
   \frac{t_{j+1}-t_j}{1-2t_j}=\lambda_j\le\frac{1}{\pi} \quad\text{ for all }\quad j\in\N_0\,.
  \end{equation}
  Moreover, considering $t<\max W(\lambda)=\max_{j\in\N_0}t_j$, there is $k\in\N_0$ such that $t_k\le t< t_{k+1}$. In particular,
  one has
  \begin{equation}\label{eq:locSeqII}
   \frac{t-t_k}{1-2t_k} < \frac{t_{k+1}-t_k}{1-2t_k} = \lambda_k\le\frac{1}{\pi}\,.
  \end{equation}
  Using the triangle inequality for the metric $\rho$ given by \eqref{eq:metric}, it follows from \eqref{eq:localPert},
  \eqref{eq:localSpecPert}, \eqref{eq:locSeqI}, \eqref{eq:locSeqII}, and Proposition \ref{prop:genRotBound} that
  \[
   \begin{aligned}
    \arcsin\bigl(\norm{P_0-P_t}\bigr)
    &\le \sum_{j=0}^{k-1}\arcsin\bigl(\norm{P_{t_j}-P_{t_{j+1}}}\bigr) + \arcsin\bigl(\norm{P_{t_k}-P_t}\bigr)\\
    &\le \sum_{j=0}^{k-1} M(\lambda_j) + M(\lambda_k) \le \sum_{j=0}^\infty M(\lambda_j)=\theta\,,
   \end{aligned}
  \]
  that is,
  \begin{equation}\label{eq:mainEstFort}
   \arcsin\bigl(\norm{P_0-P_t}\bigr) \le \theta\quad\text{ for all }\quad 0\le t < T(\theta)\,.
  \end{equation}
  Since the mapping $\bigl[0,\frac{1}{2}\bigr)\ni\tau\mapsto P_\tau$ is norm continuous and
  $S(\theta)<S\bigl(\frac{\pi}{2}\bigr)\le\frac{1}{2}$, estimate \eqref{eq:mainEstFort} also holds for $t=S(\theta)\le T(\theta)$.
  This shows \eqref{eq:mainEstForS} and, hence, completes the proof.
 \end{proof}%
\end{proposition}

It turns out that the mapping $\bigl[0,\frac{\pi}{2}\bigr]\ni\theta\mapsto T(\theta)$ is continuous and strictly increasing. It
therefore satisfies the hypotheses of Proposition \ref{prop:mainEstimate}. In this respect, it remains to compute $T(\theta)$ for
$\theta\in\bigl[0,\frac{\pi}{2}\bigr]$ in order to prove Theorem \ref{thm:mainResult}. This is done in Section \ref{sec:solOptProb}
below. For convenience, the following theorem states the corresponding result in advance.

\begin{theorem}\label{thm:solOptProb}
 In the interval $\bigl(0,\frac{\pi}{2}\bigr]$ the equation
 \[
  \Bigl(1-\frac{2}{\pi}\sin\vartheta\Bigr)^2 = \biggl(1-\frac{2}{\pi}\sin\Bigl(\frac{2\vartheta}{3}\Bigr)\biggr)^3
 \]
 has a unique solution $\vartheta\in\bigl(\arcsin\bigl(\frac{2}{\pi}\bigr),\frac{\pi}{2}\bigr)$. Moreover, the quantity $T(\theta)$
 given in \eqref{eq:optProb} has the representation
 \begin{equation}\label{eq:solOptProb}
  T(\theta) = \begin{cases}
               \frac{1}{\pi}\sin(2\theta) & \text{ for }\quad 0 \le \theta \le \arctan\bigl(\frac{2}{\pi}\bigr)=\frac{1}{2}
                  \arcsin\bigl(\frac{4\pi}{\pi^2+4}\bigr)\,,\\[0.1cm]
               \frac{2}{\pi^2}+\frac{\pi^2-4}{2\pi^2}\sin^2\theta & \text{ for }\quad \arctan\bigl(\frac{2}{\pi}
                  \bigr)<\theta<\arcsin\bigl(\frac{2}{\pi}\bigr)\,,\\[0.1cm]
               \frac{1}{2}-\frac{1}{2}\bigl(1-\frac{2}{\pi}\sin\theta\bigr)^2 & \text{ for }\quad
                  \arcsin\bigl(\frac{2}{\pi}\bigr)\le\theta\le\vartheta\,,\\[0.1cm]
               \frac{1}{2}-\frac{1}{2}\Bigl(1-\frac{2}{\pi}\sin\bigl(\frac{2\theta}{3}\bigr)\Bigr)^3 &\text{ for }\quad
                  \vartheta < \theta \le \frac{\pi}{2}\,.
              \end{cases}
 \end{equation}
 The mapping $\bigl[0,\frac{\pi}{2}\bigr]\ni\theta\mapsto T(\theta)$ is strictly increasing, continuous on
 $\bigl[0,\frac{\pi}{2}\bigr]$, and continuous differentiable on $\bigl(0,\frac{\pi}{2}\bigr)\setminus\{\vartheta\}$.
\end{theorem}

Theorem \ref{thm:mainResult} is now a straightforward consequence of Proposition \ref{prop:mainEstimate} and Theorem
\ref{thm:solOptProb}.

\begin{proof}[Proof of Theorem \ref{thm:mainResult}]
 According to Theorem \ref{thm:solOptProb}, the mapping $\bigl[0,\frac{\pi}{2}\bigr]\ni\theta\mapsto T(\theta)$ is strictly
 increasing and continuous. Hence, its range is the whole interval $[0,c_\mathrm{crit}]$, where $c_\mathrm{crit}$ is given by
 $c_\mathrm{crit}=T\bigl(\frac{\pi}{2}\bigr)=\frac{1}{2}-\frac{1}{2}\bigl(1-\frac{\sqrt{3}}{\pi}\bigl)^3$. Let
 $N=T^{-1}\colon[0,c_\mathrm{crit}]\to\bigl[0,\frac{\pi}{2}\bigr]$ denote the inverse of this mapping.
 
 Obviously, the function $N$ is also strictly increasing and continuous. Moreover, using representation \eqref{eq:solOptProb}, it
 is easy to verify that $N$ is explicitly given by \eqref{eq:mainResultFunc}. In particular, the constant
 $\kappa=T(\vartheta)=\frac{1}{2}-\frac{1}{2}\bigl(1-\frac{2}{\pi}\sin\vartheta\bigr)^2\in
 \bigl(4\frac{\pi^2-2}{\pi^4},2\frac{\pi-1}{\pi^2}\bigr)$ is the unique solution to equation \eqref{eq:kappa} in the interval
 $\bigl(0,2\frac{\pi-1}{\pi^2}\bigr]$. Furthermore, the function $N$ is continuously differentiable on
 $(0,c_\mathrm{crit})\setminus\{\kappa\}$ since the mapping $\theta\mapsto T(\theta)$ is continuously differentiable on
 $\bigl(0,\frac{\pi}{2}\bigr)\setminus\{\vartheta\}$.

 Let $V$ be a bounded self-adjoint operator on $\cH$ satisfying $\norm{V}<c_\mathrm{crit}\cdot d$. The case $V=0$ is obvious.
 Assume that $V\neq0$. Then, $B_t:=A+td\frac{V}{\norm{V}}$, $\Dom(B_t):=\Dom(A)$, and $P_t:=\EE_{B_t}\bigl(\cO_{d/2}(\sigma)\bigr)$
 for $0\le t<\frac{1}{2}$ satisfy Hypothesis \ref{app:hypHyp}. Moreover, one has $A+V=B_\tau$ with
 $\tau = \frac{\norm{V}}{d}<c_\mathrm{crit}=T\bigl(\frac{\pi}{2}\bigr)$. Applying Proposition \ref{prop:mainEstimate} to the
 mapping $\theta\mapsto T(\theta)$ finally gives
 \begin{equation}\label{eq:finalEstimate}
  \arcsin\bigl(\norm{E_A(\sigma)-E_{A+V}\bigl(\cO_{d/2}(\sigma)\bigr)}\bigr)=\arcsin\bigl(\norm{P_0-P_\tau}\bigr) \le N(\tau)
  = N\Bigl(\frac{\norm{V}}{d}\Bigr)\,,
 \end{equation}
 which completes the proof.
\end{proof}%

\begin{remark}\label{rem:kappaRepl}
 Numerical evaluations give $\vartheta=1{.}1286942\ldots<\arcsin\bigl(\frac{4\pi}{\pi^2+4}\bigr)=2\arctan\bigl(\frac{2}{\pi}\bigr)$
 and $\kappa = T(\vartheta) = 0{.}4098623\ldots<\frac{8\pi^2}{(\pi^2+4)^2}$.

 However, the estimate \eqref{eq:finalEstimate} remains valid if the constant $\kappa$ in the explicit representation for the
 function $N$ is replaced by any other constant within the interval $\bigl(4\frac{\pi^2-2}{\pi^4},2\frac{\pi-1}{\pi^2}\bigr)$. This
 can be seen by applying Proposition \ref{prop:mainEstimate} to each of the two mappings
 \[
  \theta\mapsto \frac{1}{2}-\frac{1}{2}\Bigl(1-\frac{2}{\pi}\sin\theta\Bigr)^2\quad\text{ and }\quad
  \theta\mapsto \frac{1}{2}-\frac{1}{2}\biggl(1-\frac{2}{\pi}\sin\Bigl(\frac{2\theta}{3}\Bigr)\biggr)^3\,.
 \]
 These mappings indeed satisfy the hypotheses of Proposition \ref{prop:mainEstimate}. Both are obviously continuous and strictly
 increasing, and, by particular choices of $\lambda\in D(\theta)$, it is easy to see from the considerations in Section
 \ref{sec:solOptProb} that they are less or equal to $T(\theta)$, see equation \eqref{eq:limitEquiParam} below.
\end{remark}

The statement of Theorem \ref{thm:solOptProb} actually goes beyond that of Theorem \ref{thm:mainResult}. As a matter of fact,
instead of equality in \eqref{eq:solOptProb}, it would be sufficient for the proof of Theorem \ref{thm:mainResult} to have that the
right-hand side of \eqref{eq:solOptProb} is just less or equal to $T(\theta)$. This, in turn, is rather easy to establish by
particular choices of $\lambda\in D(\theta)$, see Lemma \ref{lem:critPoints} and the proof of Lemma \ref{lem:twoParams} below.

However, Theorem \ref{thm:solOptProb} states that the right-hand side of \eqref{eq:solOptProb} provides an exact representation for
$T(\theta)$, and most of the considerations in Section \ref{sec:solOptProb} are required to show this stronger result. As a
consequence, the bound from Theorem \ref{thm:mainResult} is optimal within the framework of the approach by estimate
\eqref{eq:essEstimate}.

In fact, the following observation shows that a bound substantially stronger than the one from Proposition \ref{prop:genRotBound}
is required, at least for small perturbations, in order to improve on Theorem \ref{thm:mainResult}.

\begin{remark}
 One can modify the approach \eqref{eq:essEstimate} by replacing the term $M(\lambda_j)=\frac{1}{2}\arcsin(\pi\lambda_j)$ by
 $N(\lambda_j)$ and relaxing the condition \eqref{eq:seqParamCond} to $\lambda_j\le c_{\text{crit}}$. Yet, it follows from Theorem
 \ref{thm:solOptProb} that the corresponding optimization procedure leads to exactly the same result \eqref{eq:solOptProb}. This
 can be seen from the fact that each $N(\lambda_j)$ is of the form of the right-hand side of \eqref{eq:essEstimate} (cf.\ the
 computation of $T(\theta)$ in Section \ref{sec:solOptProb} below), so that we are actually dealing with essentially the same
 optimization problem. In this sense, the function $N$ is a fixed point in the approach presented here.
\end{remark}

We close this section with a comparison of Theorem \ref{thm:mainResult} with the strongest previously known result by Albeverio and
Motovilov from \cite{AM13}.

\begin{remark}\label{rem:estOptimality}
 One has $N(x)=M_*(x)$ for $0\le x\le \frac{4}{\pi^2+4}$, and the inequality $N(x)<M_*(x)$ holds for all
 $\frac{4}{\pi^2+4}<x\le c_*$, where $c_*\in\bigl(0,\frac{1}{2}\bigr)$ and $M_*\colon[0,c_*]\to\bigl[0,\frac{\pi}{2}\bigr]$ are
 given by \eqref{eq:AMConst} and \eqref{eq:AMFunc}, respectively. Indeed, it follows from the computation of $T(\theta)$ in Section
 \ref{sec:solOptProb} (see Remark \ref{rem:AMvsS} below) that
 \[
  x < T(M_*(x))\le c_\mathrm{crit}\quad\text{ for }\quad \frac{4}{\pi^2+4} < x \le c_*\,.
 \]
 Since the function $N=T^{-1}\colon[0,c_{\mathrm{crit}}]\to\bigl[0,\frac{\pi}{2}\bigr]$ is strictly increasing, this implies that
 \[
  N(x)< N\bigl(T(M_*(x))\bigr)=M_*(x)\quad \text{ for }\quad \frac{4}{\pi^2+4}<x\le c_*\,.
 \]
\end{remark}

%%%%%%%%%%%%%%%%%%%%%%%%%%%%%%%%%%%%%%%%%%%%%%%%%%%%%%%%%%%%%%%%%%%%%%%%%%%%%%%%%%%%%%%%%%%%%%%%%%%%%%%%%%%%%%%%%%%%%%%%%%%%%%%%%%%
%%%%%%%%%%%%%%%%%%%%%%%%%%%%%%%%%%%%%%%%%%%%%%%%%%%%%%%%%%%%%%%%%%%%%%%%%%%%%%%%%%%%%%%%%%%%%%%%%%%%%%%%%%%%%%%%%%%%%%%%%%%%%%%%%%%
%%%% Solution of the optimization problem
%%%%%%%%%%%%%%%%%%%%%%%%%%%%%%%%%%%%%%%%%%%%%%%%%%%%%%%%%%%%%%%%%%%%%%%%%%%%%%%%%%%%%%%%%%%%%%%%%%%%%%%%%%%%%%%%%%%%%%%%%%%%%%%%%%%
%%%%%%%%%%%%%%%%%%%%%%%%%%%%%%%%%%%%%%%%%%%%%%%%%%%%%%%%%%%%%%%%%%%%%%%%%%%%%%%%%%%%%%%%%%%%%%%%%%%%%%%%%%%%%%%%%%%%%%%%%%%%%%%%%%%
\section{Proof of Theorem \ref{thm:solOptProb}}\label{sec:solOptProb}
We split the proof of Theorem \ref{thm:solOptProb} into several steps. We first reduce the problem of computing $T(\theta)$ to the
problem of solving suitable finite-dimensional constrained optimization problems, see equations \eqref{eq:supOptn} and
\eqref{eq:extremalProb}. The corresponding critical points are then characterized in Lemma \ref{lem:critPoints} using Lagrange
multipliers. The crucial tool to reduce the set of relevant critical points is provided by Lemma \ref{lem:paramSubst}. Finally, the
finite-dimensional optimization problems are solved in Lemmas \ref{lem:twoParams}, \ref{lem:threeParams}, and
\ref{prop:solOptProb}.

Throughout this section, we make use of the notations introduced in Definitions \ref{def:params} and \ref{def:optProb}. In
addition, we fix the following notations.

\begin{definition}
 For $n\in\N_0$ and $\theta\in\bigl[0,\frac{\pi}{2}\bigr]$ define $D_n(\theta):=D(\theta)\cap D_n$. Moreover, let
 \[
  T_n(\theta):=\sup\bigl\{\max W(\lambda) \bigm| \lambda\in D_n(\theta)\bigr\}\quad\text{ if }\quad D_n(\theta)\neq\emptyset\,,
 \]
 and set $T_n(\theta):=0$ if $D_n(\theta)=\emptyset$.
\end{definition}

As a result of $D(0)=D_n(0)=\{0\}\subset l^1(\N_0)$, we have $T(0)=T_n(0)=0$ for every $n\in\N_0$. Let
$\theta\in\bigl(0,\frac{\pi}{2}\bigr]$ be arbitrary. Since $D_0(\theta) \subset D_1(\theta) \subset D_2(\theta)\subset \dots$, we
obtain
\[
 T_0(\theta) \le T_1(\theta) \le T_2(\theta) \le \dots
\]
Moreover, we observe that
\begin{equation}\label{eq:supOptn}
 T(\theta)=\sup_{n\in\N_0}T_n(\theta)\,.
\end{equation}
In fact, we show below that $T_n(\theta)=T_2(\theta)$ for every $n\ge 2$, so that $T(\theta)=T_2(\theta)$, see Lemma
\ref{prop:solOptProb}.

Let $n\in\N$ be arbitrary and let $\lambda=(\lambda_j)\in D_n$. Denote $(t_j):=W(\lambda)$. It follows from part (b) of Lemma
\ref{lem:seq} that $\max W(\lambda)=t_{n+1}$. Moreover, we have
\[
 1-2t_{j+1} = 1-2t_j -2\lambda_j(1-2t_j) = (1-2t_j)(1-2\lambda_j)\,,\quad j=0,\dots,n\,.
\]
Since $t_0=0$, this implies that
\[
 1-2t_{n+1} = \prod_{j=0}^n (1-2\lambda_j)\,.
\]
In particular, we obtain the explicit representation
\begin{equation}\label{eq:seqExplRepr}
 \max W(\lambda)=t_{n+1} = \frac{1}{2}\biggl(1-\prod_{j=0}^n(1-2\lambda_j)\biggr)\,.
\end{equation}

An immediate conclusion of representation \eqref{eq:seqExplRepr} is the following statement.

\begin{lemma}\label{lem:paramPermut}
 For $\lambda=(\lambda_j)\in D_n$ the value of $\max W(\lambda)$ does not depend on the order of the entries
 $\lambda_0,\dots,\lambda_n$.
\end{lemma}

Another implication of representation \eqref{eq:seqExplRepr} is the fact that $\max W(\lambda)=t_{n+1}$ can be considered as a
continuous function of the variables $\lambda_0,\dots,\lambda_n$. Since the set $D_n(\theta)$ is compact as a closed bounded subset
of an $(n+1)$-dimensional subspace of $l^1(\N_0)$, we deduce that $T_n(\theta)$ can be written as
\begin{equation}\label{eq:extremalProb}
 T_n(\theta) = \max\bigl\{t_{n+1} \bigm| (t_j)=W(\lambda)\,,\ \lambda\in D_n(\theta)\bigr\}\,.
\end{equation}
Hence, $T_n(\theta)$ is determined by a finite-dimensional constrained optimization problem, which can be studied by use of
Lagrange multipliers.

Taking into account the definition of the set $D_n(\theta)$, it follows from equation \eqref{eq:extremalProb} and representation
\eqref{eq:seqExplRepr} that there is some point $(\lambda_0,\dots,\lambda_n)\in\bigl[0,\frac{1}{\pi}\bigr]^{n+1}$ such that
\[
 T_n(\theta) = t_{n+1} = \frac{1}{2}\biggl(1-\prod_{j=0}^n(1-2\lambda_j)\biggr)\quad\text{ and }\quad
 \sum_{j=0}^n M(\lambda_j)=\theta\,,
\]
where $M(x)=\frac{1}{2}\arcsin(\pi x)$ for $0\le x\le\frac{1}{\pi}$. In particular, if
$(\lambda_0,\dots,\lambda_n)\in\bigl(0,\frac{1}{\pi}\bigr)^{n+1}$, then the method of Lagrange multipliers gives a constant
$r\in\R$, $r\neq0$, with
\[
 \frac{\partial t_{n+1}}{\partial\lambda_k} = r\cdot M'(\lambda_k) = r\cdot\frac{\pi}{2\sqrt{1-\pi^2\lambda_k^2}}
 \quad\text{ for }\quad k=0,\dots,n\,.
\]
Hence, in this case, for every $k\in\{0,\dots,n-1\}$ we obtain
\begin{equation}\label{eq:critPointsCond}
 \frac{\sqrt{1-\pi^2\lambda_k^2}}{\sqrt{1-\pi^2\lambda_{k+1}^2}} =
 \frac{\dfrac{\partial t_{n+1}}{\partial\lambda_{k+1}}}{\dfrac{\partial t_{n+1}}{\partial\lambda_k}} =
 \frac{\prod\limits_{\substack{j=0\\ j\neq k+1}}^n (1-2\lambda_j)}{\prod\limits_{\substack{j=0\\ j\neq k}}^n(1-2\lambda_j)} =
 \frac{1-2\lambda_k}{1-2\lambda_{k+1}}\,.
\end{equation}
This leads to the following characterization of critical points of the mapping $\lambda\mapsto\max W(\lambda)$ on $D_n(\theta)$.

\begin{lemma}\label{lem:critPoints}
 For $n\ge 1$ and $\theta\in\bigl(0,\frac{\pi}{2}\bigr]$ let $\lambda=(\lambda_j)\in D_n(\theta)$ with
 $T_n(\theta)=\max W(\lambda)$. Assume that $\lambda_0\ge\dots\ge\lambda_n$. If, in addition, $\lambda_0<\frac{1}{\pi}$ and
 $\lambda_n>0$, then either one has
 \[
  \lambda_0=\dots=\lambda_n=\frac{1}{\pi}\sin\Bigl(\frac{2\theta}{n+1}\Bigr)\,,
 \]
 so that
 \begin{equation}\label{eq:limitEquiParam}
  \max W(\lambda) = \frac{1}{2}-\frac{1}{2}\left(1-\frac{2}{\pi}\sin\Bigl(\frac{2\theta}{n+1}\Bigr)\right)^{n+1}\,,
 \end{equation}
 or there is $l\in\{0,\dots,n-1\}$ with
 \begin{equation}\label{eq:critPointsNotEqual}
  \frac{4}{\pi^2+4}>\lambda_0=\dots=\lambda_l>\frac{2}{\pi^2}>\lambda_{l+1}=\dots=\lambda_n>0\,.
 \end{equation}
 In the latter case, $\lambda_0$ and $\lambda_n$ satisfy
 \begin{equation}\label{eq:critPointsRels}
  \lambda_0+\lambda_n =\frac{4\alpha^2}{\pi^2+4\alpha^2}\quad\text{ and }\quad
  \lambda_0\lambda_n=\frac{\alpha^2-1}{\pi^2+4\alpha^2}\,,
 \end{equation}
 where
 \begin{equation}\label{eq:critPointsAlpha}
  \alpha = \frac{\sqrt{1-\pi^2\lambda_0^2}}{1-2\lambda_0} = \frac{\sqrt{1-\pi^2\lambda_n^2}}{1-2\lambda_n}\in(1,m)\,,\quad
  m:=\frac{\pi}{2}\tan\Bigl(\arcsin\Bigl(\frac{2}{\pi}\Bigr)\Bigr)\,.
 \end{equation}
 
 \begin{proof}
  Let $\lambda_0<\frac{1}{\pi}$ and $\lambda_n>0$. In particular, one has
  $(\lambda_0,\dots,\lambda_n)\in\bigl(0,\frac{1}{\pi}\bigr)^{n+1}$. Hence, it follows from \eqref{eq:critPointsCond} that
  \begin{equation}\label{eq:defAlpha}
   \alpha:=\frac{\sqrt{1-\pi^2\lambda_k^2}}{1-2\lambda_k} 
  \end{equation}
  does not depend on $k\in\{0,\dots,n\}$.
  
  If $\lambda_0=\lambda_n$, then all $\lambda_j$ coincide and one has
  $\theta=(n+1)M(\lambda_0)=\frac{n+1}{2}\arcsin(\pi\lambda_0)$, that is,
  $\lambda_0=\dots=\lambda_n=\frac{1}{\pi}\sin\bigl(\frac{2\theta}{n+1}\bigr)$. Inserting this into representation
  \eqref{eq:seqExplRepr} yields equation \eqref{eq:limitEquiParam}.
  
  Now assume that $\lambda_0>\lambda_n$. A straightforward calculation shows that $x=\frac{2}{\pi^2}$ is the only critical point of
  the mapping
  \begin{equation}\label{eq:constrParamMapping}
   \Bigl[0,\frac{1}{\pi}\Bigr]\ni x\mapsto \frac{\sqrt{1-\pi^2x^2}}{1-2x}\,,
  \end{equation}
  cf.\ \figurename\ \ref{fig:functionPlot}. The image of this point is $\bigl(1-\frac{4}{\pi^2}\bigr)^{-1/2}=m>1$. Moreover, $0$
  and $\frac{4}{\pi^2+4}$ are mapped to $1$, and $\frac{1}{\pi}$ is mapped to $0$. In particular, every value in the interval
  $(1,m)$ has exactly two preimages under the mapping \eqref{eq:constrParamMapping}, and all the other values in the range $[0,m]$
  have only one preimage. Since $\lambda_0>\lambda_n$ by assumption, it follows from \eqref{eq:defAlpha} that $\alpha$ has two
  preimages. Hence, $\alpha\in(1,m)$ and $\frac{4}{\pi^2+4}>\lambda_0>\frac{2}{\pi^2}>\lambda_n>0$. Furthermore, there is
  $l\in\{0,\dots,n-1\}$ with $\lambda_0=\dots=\lambda_l$ and $\lambda_{l+1}=\dots=\lambda_n$. This proves
  \eqref{eq:critPointsNotEqual} and \eqref{eq:critPointsAlpha}.

  Finally, the relations \eqref{eq:critPointsRels} follow from the fact that the equation $\frac{\sqrt{1-\pi^2z^2}}{1-2z}=\alpha$
  can be rewritten as
  \[
   0 = z^2 - \frac{4\alpha^2}{\pi^2+4\alpha^2}\,z + \frac{\alpha^2-1}{\pi^2+4\alpha^2} = (z-\lambda_0)(z-\lambda_n)
   = z^2 - (\lambda_0+\lambda_n)z + \lambda_0\lambda_n\,.\qedhere
  \]
 \end{proof}%
\end{lemma}

\begin{figure}[ht]\begin{center} 
 \scalebox{0.5}{\includegraphics{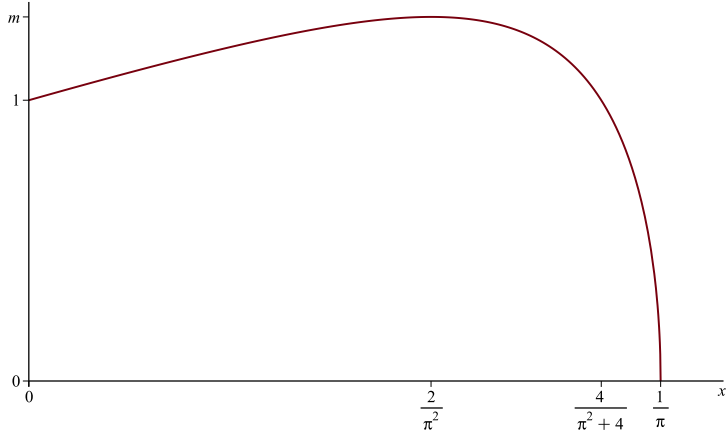}}
 \caption{The mapping $\bigl[0,\frac{1}{\pi}\bigr]\ni x\mapsto \frac{\sqrt{1-\pi^2x^2}}{1-2x}$.}
 \label{fig:functionPlot}
\end{center}
\end{figure}

The preceding lemma is one of the main ingredients for solving the constrained optimization problem that defines the quantity
$T_n(\theta)$ in \eqref{eq:extremalProb}. However, it is still a hard task to compute $T_n(\theta)$ from the corresponding critical
points. Especially the case \eqref{eq:critPointsNotEqual} in Lemma \ref{lem:critPoints} is difficult to handle and needs careful
treatment. An efficient computation of $T_n(\theta)$ therefore requires a technique that allows to narrow down the set of relevant
critical points. The following result provides an adequate tool for this and is thus crucial for the remaining considerations. The
idea behind this approach may also prove useful for solving similar optimization problems.

\begin{lemma}\label{lem:paramSubst}
 For $n\ge 1$ and $\theta\in\bigl(0,\frac{\pi}{2}\bigr]$ let $\lambda=(\lambda_j)\in D_n(\theta)$. If
 $T_n(\theta)=\max W(\lambda)$, then for every $k\in\{0,\dots,n\}$ one has
 \[
  \max W\bigl((\lambda_0,\dots,\lambda_k,0,\dots)\bigr) = T_k(\theta_k)\quad\text{ with }\quad
  \theta_k = \sum_{j=0}^k M(\lambda_j) \le \theta\,.
 \]
 
 \begin{proof}
  Suppose that $T_n(\theta)=\max W(\lambda)$. The case $k=n$ in the claim obviously agrees with this hypothesis.

  Let $k\in\{0,\dots,n-1\}$ be arbitrary and denote $(t_j):=W(\lambda)$. It follows from part (b) of Lemma \ref{lem:seq} that
  $t_{k+1} = \max W\bigl( (\lambda_0,\dots,\lambda_k,0,\dots) \bigr)$. In particular, one has $t_{k+1}\le T_k(\theta_k)$ since
  $(\lambda_0,\dots,\lambda_k,0,\dots)\in D_k(\theta_k)$.

  Assume that $t_{k+1}<T_k(\theta_k)$, and let $\gamma=(\gamma_j)\in D_k(\theta_k)$ with $\max W(\gamma)=T_k(\theta_k)$. Denote
  $\mu:=(\gamma_0,\dots,\gamma_k,\lambda_{k+1},\dots,\lambda_n,0,\dots)\in D_n(\theta_n)$ and $(s_j):=W(\mu)$. Again by part (b) of
  Lemma \ref{lem:seq}, one has $s_{k+1}=\max W(\gamma)>t_{k+1}$ and $s_{n+1}=\max W(\mu)\le T_n(\theta_n)$. Taking into account
  part (a) of Lemma \ref{lem:seq} and the definition of the operator $W$, one obtains that
  \[
   t_{k+2} = t_{k+1} + \lambda_{k+1}(1-2t_{k+1}) < s_{k+1} + \lambda_{k+1}(1-2s_{k+1}) = s_{k+2}\,.
  \]
  Iterating this estimate eventually gives $t_{n+1}<s_{n+1}\le T_n(\theta_n)$, which contradicts the case $k=n$ from above. Thus,
  $\max W\bigl((\lambda_0,\dots,\lambda_k,0,\dots)\bigr)=t_{k+1}=T_k(\theta_k)$ as claimed.
 \end{proof}%
\end{lemma}

Lemma \ref{lem:paramSubst} states that if a sequence $\lambda\in D_n(\theta)$ solves the optimization problem for $T_n(\theta)$,
then every truncation of $\lambda$ solves the corresponding reduced optimization problem. This allows to exclude many sequences in
$D_n(\theta)$ from the considerations once the optimization problem is understood for small $n$. The number of parameters in
\eqref{eq:extremalProb} can thereby be reduced considerably.

The following lemma demonstrates this technique. It implies that the condition $\lambda_0<\frac{1}{\pi}$ in Lemma
\ref{lem:critPoints} is always satisfied except for one single case, which can be treated separately.

\begin{lemma}\label{lem:lambda_0}
 For $n\ge 1$ and $\theta\in\bigl(0,\frac{\pi}{2}\bigr]$ let $\lambda=(\lambda_j)\in D_n(\theta)$ with
 $T_n(\theta)=\max W(\lambda)$ and $\lambda_0\ge\dots\ge\lambda_n$. If $\lambda_0=\frac{1}{\pi}$, then $\theta=\frac{\pi}{2}$ and
 $n=1$.
 
 \begin{proof}
  Let $\lambda_0=\frac{1}{\pi}$ and define $\theta_1:= M(\lambda_0)+M(\lambda_1)\le\theta$. It is obvious that
  $\lambda\in D_1\bigl(\frac{\pi}{2})$ is equivalent to $\theta_1=\theta=\frac{\pi}{2}$. Assume that $\theta_1<\frac{\pi}{2}$.
  Clearly, one has $\theta_1\ge M(\lambda_0)=\frac{\pi}{4}$ and $\lambda_1 = \frac{1}{\pi}\sin\bigl(2\theta_1-\frac{\pi}{2}\bigr) =
  -\frac{1}{\pi}\left(1-2\sin^2\theta_1\right)\in\bigl[0,\frac{1}{\pi}\bigr)$. Taking into account representation
  \eqref{eq:seqExplRepr}, for $\mu:=(\lambda_0,\lambda_1,0,\dots)\in D_1(\theta_1)$ one computes
  \[
   \begin{aligned}
    \max W(\mu) &= \frac{1}{2}-\frac{1}{2}(1-2\lambda_0)(1-2\lambda_1) = (\lambda_0+\lambda_1) - 2\lambda_0\lambda_1\\
                &= \frac{2}{\pi}\sin^2\theta_1+\frac{2}{\pi^2}\left(1-2\sin^2\theta_1\right)=\frac{2}{\pi^2} + \frac{2\pi-4}
                   {\pi^2}\sin^2\theta_1\,.
   \end{aligned}
  \]
  Since $\arcsin\bigl(\frac{1}{\pi-1}\bigr)<\frac{\pi}{4}\le\theta_1<\frac{\pi}{2}$, it follows from part (\ref{app:lem1:a}) of
  Lemma \ref{app1:lem1} that
  \[
   \max W(\mu) < \frac{2}{\pi}\Bigl(1-\frac{1}{\pi}\sin\theta_1\Bigr)\sin\theta_1 =
   \frac{1}{2}-\frac{1}{2}\Bigl(1-\frac{2}{\pi}\sin\theta_1\Bigr)^2 \le T_1(\theta_1)\,,
  \]
  where the last inequality is due to representation \eqref{eq:limitEquiParam}. This is a contradiction to Lemma
  \ref{lem:paramSubst}. Hence, $\theta_1=\theta=\frac{\pi}{2}$ and, in particular, $\lambda=\mu\in D_1\bigl(\frac{\pi}{2}\bigr)$.
  
  Obviously, one has $D_1\bigl(\frac{\pi}{2}\bigr)=\left\{\bigl(\frac{1}{\pi},\frac{1}{\pi},0,\dots\bigr)\right\}$, so that
  $\lambda=\bigl(\frac{1}{\pi},\frac{1}{\pi},0,\dots\bigr)$. Taking into account that
  $\sin\bigl(\frac{\pi}{3}\bigr)=\frac{\sqrt{3}}{2}$, it follows from representations \eqref{eq:seqExplRepr} and
  \eqref{eq:limitEquiParam} that
  \[
   \max W(\lambda) = \frac{1}{2}-\frac{1}{2}\Bigl(1-\frac{2}{\pi}\Bigr)^2 <
   \frac{1}{2}-\frac{1}{2}\biggl(1-\frac{\sqrt{3}}{\pi}\biggr)^3\le T_2\Bigl(\frac{\pi}{2}\Bigr)\,.
  \]
  Since $\max W(\lambda)=T_n(\theta)$ by hypothesis, this implies that $n=1$.
 \end{proof}%
\end{lemma}

We are now able to solve the finite-dimensional constrained optimization problem in \eqref{eq:extremalProb} for every
$\theta\in\bigl[0,\frac{\pi}{2}\bigr]$ and $n\in\N$. We start with the case $n=1$.

\begin{lemma}\label{lem:twoParams}
 The quantity $T_1(\theta)$ has the representation
 \[
  T_1(\theta) = \begin{cases}
                  T_0(\theta)=\frac{1}{\pi}\,\sin(2\theta) & \text{ for }\quad
                    0\le\theta\le\arctan\bigl(\frac{2}{\pi}\bigr)=\frac{1}{2}\arcsin\bigl(\frac{4\pi}
                    {\pi^2+4}\bigr)\,,\\[0.1cm]
                  \frac{2}{\pi^2}+\frac{\pi^2-4}{2\pi^2}\sin^2\theta & \text{ for }\quad
                    \arctan\bigl(\frac{2}{\pi}\bigr)<\theta<\arcsin\bigl(\frac{2}{\pi}\bigr)\,,\\[0.1cm]
                  \frac{1}{2}-\frac{1}{2}\bigl(1-\frac{2}{\pi}\sin\theta\bigr)^2 & \text{ for }\quad
                    \arcsin\bigl(\frac{2}{\pi}\bigr)\le\theta\le\frac{\pi}{2}\,.
                \end{cases}
 \]
 In particular, if $0<\theta<\arcsin\bigl(\frac{2}{\pi}\bigr)$ and $\lambda=(\lambda_0,\lambda_1,0,\dots)\in D_1(\theta)$ with
 $\lambda_0=\lambda_1$, then the strict inequality $\max W(\lambda)<T_1(\theta)$ holds.

 The mapping $\bigl[0,\frac{\pi}{2}\bigr]\ni\theta\mapsto T_1(\theta)$ is strictly increasing, continuous on
 $\bigl[0,\frac{\pi}{2}\bigr]$, and continuously differentiable on $\bigl(0,\frac{\pi}{2}\bigr)$.
 
 \begin{proof}
  Since $T_1(0)=T_0(0)=0$, the representation is obviously correct for $\theta=0$. For $\theta=\frac{\pi}{2}$ one has
  $D_1\bigl(\frac{\pi}{2}\bigr)=\left\{\bigl(\frac{1}{\pi},\frac{1}{\pi},0,\dots\bigr)\right\}$, so that
  $T_1\bigl(\frac{\pi}{2}\bigr)=\frac{1}{2}-\frac{1}{2}\bigl(1-\frac{2}{\pi}\bigr)^2$ by representation \eqref{eq:seqExplRepr}.
  This also agrees with the claim.
  
  Now let $\theta\in\bigl(0,\frac{\pi}{2}\bigr)$ be arbitrary. Obviously, one has
  $D_0(\theta)=\left\{\bigl(\frac{1}{\pi}\sin(2\theta),0,\dots\bigr)\right\}$ if $\theta\le\frac{\pi}{4}$, and
  $D_0(\theta)=\emptyset$ if $\theta>\frac{\pi}{4}$. Hence,
  \begin{equation}\label{eq:oneParam}
   T_0(\theta)=\frac{1}{\pi}\sin(2\theta)\quad\text{ if }\quad 0<\theta\le\frac{\pi}{4}\,,
  \end{equation}
  and $T_0(\theta)=0$ if $\theta>\frac{\pi}{4}$.
  
  By Lemmas \ref{lem:paramPermut}, \ref{lem:critPoints}, and \ref{lem:lambda_0} there are only two sequences in
  $D_1(\theta)\setminus D_0(\theta)$ that need to be considered in order to compute $T_1(\theta)$. One of them is given by
  $\mu=(\mu_0,\mu_1,0,\dots)$ with $\mu_0=\mu_1=\frac{1}{\pi}\sin\theta\in\bigl(0,\frac{1}{\pi}\bigr)$. For this sequence,
  representation \eqref{eq:limitEquiParam} yields
  \begin{equation}\label{eq:twoParamsEqui}
   \max W(\mu) = \frac{1}{2}-\frac{1}{2}\Bigl(1-\frac{2}{\pi}\sin\theta\Bigr)^2=
   \frac{2}{\pi}\Bigl(1-\frac{1}{\pi}\sin\theta\Bigr)\sin\theta\,.
  \end{equation}
  
  The other sequence in $D_1(\theta)\setminus D_0(\theta)$ that needs to be considered is $\lambda=(\lambda_0,\lambda_1,0,\dots)$
  with $\lambda_0$ and $\lambda_1$ satisfying $\frac{4}{\pi^2+4}>\lambda_0>\frac{2}{\pi^2}>\lambda_1>0$ and
  \begin{equation}\label{eq:twoParamsRels}
   \lambda_0 + \lambda_1 = \frac{4\alpha^2}{\pi^2+4\alpha^2}\,,\quad
   \lambda_0\lambda_1 = \frac{\alpha^2-1}{\pi^2+4\alpha^2}\,,
  \end{equation}
  where
  \begin{equation}\label{eq:twoParamsAlpha}
   \alpha = \frac{\sqrt{1-\pi^2\lambda_0^2}}{1-2\lambda_0} = \frac{\sqrt{1-\pi^2\lambda_1^2}}{1-2\lambda_1}\in(1,m)\,,\quad
   m=\frac{\pi}{2}\tan\Bigl(\arcsin\Bigl(\frac{2}{\pi}\Bigr)\Bigr)\,.
  \end{equation}
  It turns out shortly that this sequence $\lambda$ exists if and only if
  $\arctan\bigl(\frac{2}{\pi}\bigr)<\theta<\arcsin\bigl(\frac{2}{\pi}\bigr)$.
  
  Using representation \eqref{eq:seqExplRepr} and the relations in \eqref{eq:twoParamsRels}, one obtains
  \begin{equation}\label{eq:twoParamsLimitAlpha}
   \max W(\lambda) = \frac{1}{2}-\frac{1}{2}(1-2\lambda_0)(1-2\lambda_1) = (\lambda_0+\lambda_1) - 2\lambda_0\lambda_1
   = 2\,\frac{\alpha^2+1}{\pi^2+4\alpha^2}\,.
  \end{equation}
  The objective is to rewrite the right-hand side of \eqref{eq:twoParamsLimitAlpha} in terms of $\theta$.
  
  It follows from
  \begin{equation}\label{eq:twoParams2Theta}
   2\theta = \arcsin(\pi\lambda_0) + \arcsin(\pi\lambda_1)
  \end{equation}
  and the relations \eqref{eq:twoParamsRels} and \eqref{eq:twoParamsAlpha} that
  \begin{equation}\label{eq:twoParamsSin2ThetaInAlpha}
   \begin{aligned}
    \sin(2\theta)
     &= \pi\lambda_0\sqrt{1-\pi^2\lambda_1^2} + \pi\lambda_1\sqrt{1-\pi^2\lambda_0^2}
      = \alpha\pi\lambda_0(1-2\lambda_1) + \alpha\pi\lambda_1(1-2\lambda_0)\\
     &= \alpha\pi\left(\lambda_0+\lambda_1 - 4\lambda_0\lambda_1\right)
      = \frac{4\alpha\pi}{\pi^2+4\alpha^2}\,.
   \end{aligned}
  \end{equation}
  Taking into account that $\sin(2\theta)>0$, equation \eqref{eq:twoParamsSin2ThetaInAlpha} can be rewritten as
  \[
   \alpha^2 - \frac{\pi}{\sin(2\theta)}\alpha + \frac{\pi^2}{4}=0\,.
  \]
  In turn, this gives
  \[
   \alpha = \frac{\pi}{2\sin(2\theta)}\left(1\pm\sqrt{1-\sin^2(2\theta)}\right)
   = \frac{\pi}{2}\,\frac{1\pm \abs{\cos^2\theta-\sin^2\theta}}{2\sin\theta\cos\theta}\,,
  \]
  that is,
  \begin{equation}\label{eq:twoParamsPossAlpha}
   \alpha = \frac{\pi}{2}\tan\theta\quad\text{ or }\quad \alpha = \frac{\pi}{2}\cot\theta\,.
  \end{equation}
  We show that the second case in \eqref{eq:twoParamsPossAlpha} does not occur.

  Since $1<\alpha<m<\frac{\pi}{2}$, by equation \eqref{eq:twoParamsSin2ThetaInAlpha} one has $\sin(2\theta)<1$, which implies that
  $\theta\neq\frac{\pi}{4}$. Moreover, combining relations \eqref{eq:twoParamsRels} and \eqref{eq:twoParamsAlpha}, $\lambda_1$ can
  be expressed in terms of $\lambda_0$ alone. Hence, by equation \eqref{eq:twoParams2Theta} the quantity $\theta$ can be written as
  a continuous function of the sole variable $\lambda_0\in \bigl(\frac{2}{\pi^2},\frac{4}{\pi^2+4}\bigr)$. Taking the limit
  $\lambda_0\to\frac{4}{\pi^2+4}$ in equation \eqref{eq:twoParams2Theta} then implies that $\lambda_1\to0$ and, therefore,
  $\theta\to\frac{1}{2}\arcsin\bigl(\frac{4\pi}{\pi^2+4}\bigr)<\frac{\pi}{4}$. This yields $\theta<\frac{\pi}{4}$ for every
  $\lambda_0\in\bigl(\frac{2}{\pi^2},\frac{4}{\pi^2+4}\bigr)$ by continuity, that is, the sequence $\lambda$ can exist only if
  $\theta<\frac{\pi}{4}$. Taking into account that $\alpha$ satisfies
  $1<\alpha<m=\frac{\pi}{2}\tan\bigl(\arcsin\bigl(\frac{2}{\pi}\bigr)\bigr)$, it now follows from \eqref{eq:twoParamsPossAlpha}
  that the sequence $\lambda$ exists if and only if $\arctan\bigl(\frac{2}{\pi}\bigr)<\theta<\arcsin\bigl(\frac{2}{\pi}\bigr)$ and,
  in this case, one has
  \begin{equation}\label{eq:twoParamsAlphaInTheta}
   \alpha=\frac{\pi}{2}\tan\theta\,.
  \end{equation}

  Combining equations \eqref{eq:twoParamsLimitAlpha} and \eqref{eq:twoParamsAlphaInTheta} finally gives
  \begin{equation}\label{eq:twoParamsLimitTheta}
   \begin{aligned}
    \max W(\lambda)
    &= \frac{1}{2}\,\frac{\frac{4}{\pi^2}+\tan^2\theta}{1+\tan^2\theta} = \frac{2}{\pi^2}\cos^2\theta + \frac{1}{2}\sin^2\theta
    = \frac{2}{\pi^2} + \frac{\pi^2-4}{2\pi^2}\,\sin^2\theta
   \end{aligned}
  \end{equation}
  for $\arctan\bigl(\frac{2}{\pi}\bigr)<\theta<\arcsin\bigl(\frac{2}{\pi}\bigr)$.
  
  As a result of Lemmas \ref{lem:paramPermut}, \ref{lem:critPoints}, and \ref{lem:lambda_0}, the quantities \eqref{eq:oneParam},
  \eqref{eq:twoParamsEqui}, and \eqref{eq:twoParamsLimitTheta} are the only possible values for $T_1(\theta)$, and we have to
  determine which of them is the greatest.
  
  The easiest case is $\theta>\frac{\pi}{4}$ since then \eqref{eq:twoParamsEqui} is the only possibility for $T_1(\theta)$.
  
  The quantity \eqref{eq:twoParamsLimitTheta} is relevant only if
  $\arctan\bigl(\frac{2}{\pi}\bigr)<\theta<\arcsin\bigl(\frac{2}{\pi}\bigr)<\frac{\pi}{4}$. In this case, it follows from parts
  (\ref{app:lem1:b}) and (\ref{app:lem1:c}) of Lemma \ref{app1:lem1} that \eqref{eq:twoParamsLimitTheta} gives the greatest value
  of the three possibilities and, hence, is the correct term for $T_1(\theta)$ here.
  
  For $0<\theta\le\arctan\bigl(\frac{2}{\pi}\bigr)<2\arctan\bigl(\frac{1}{\pi}\bigr)$, by part (\ref{app:lem1:d}) of Lemma
  \ref{app1:lem1} the quantity \eqref{eq:oneParam} is greater than \eqref{eq:twoParamsEqui}. Therefore, $T_1(\theta)$ is given by
  \eqref{eq:oneParam} in this case.
  
  Finally, consider the case $\arcsin\bigl(\frac{2}{\pi}\bigr)\le\theta\le\frac{\pi}{4}$. Since
  $2\arctan\bigl(\frac{1}{\pi}\bigr)<\arcsin\bigl(\frac{2}{\pi}\bigr)$, it follows from part (\ref{app:lem1:e}) of Lemma
  \ref{app1:lem1} that \eqref{eq:twoParamsEqui} is greater than \eqref{eq:oneParam} and, hence, coincides with $T_1(\theta)$.
  
  This completes the computation of $T_1(\theta)$ for $\theta\in\bigl[0,\frac{\pi}{2}\bigr]$. In particular, it follows from the
  discussion of the two cases $0<\theta\le\arctan\bigl(\frac{2}{\pi}\bigr)$ and
  $\arctan\bigl(\frac{2}{\pi}\bigr)<\theta<\arcsin\bigl(\frac{2}{\pi}\bigr)$ that $\max W(\mu)$ is always strictly less than
  $T_1(\theta)$ if $0<\theta<\arcsin\bigl(\frac{2}{\pi}\bigr)$.
  
  The piecewise defined mapping $\bigl[0,\frac{\pi}{2}\bigr]\ni\theta\mapsto T_1(\theta)$ is continuously differentiable on each of
  the corresponding subintervals. It remains to prove that the mapping is continuous and continuously differentiable at the points
  $\theta=\arctan\bigl(\frac{2}{\pi}\bigr)=\frac{1}{2}\arcsin\bigl(\frac{4\pi}{\pi^2+4}\bigr)$ and
  $\theta=\arcsin\bigl(\frac{2}{\pi}\bigr)$.
  
  Taking into account that $\sin^2\theta=\frac{4}{\pi^2+4}$ for $\theta=\frac{1}{2}\arcsin\bigl(\frac{4\pi}{\pi^2+4}\bigr)$, the
  continuity is straightforward to verify. The continuous differentiability follows from the relations
  \[
   \frac{\pi^2-4}{\pi^2}\sin\theta\cos\theta = \frac{2}{\pi}\Bigl(1-\frac{2}{\pi}\sin\theta\Bigr)\cos\theta
   \quad\text{ for }\quad \theta=\arcsin\Bigl(\frac{2}{\pi}\Bigr)\,
  \]
  and
  \[
   \frac{2}{\pi} \cos(2\theta) = \frac{\pi^2-4}{2\pi^2}\sin(2\theta) = \frac{\pi^2-4}{\pi^2}\sin\theta\cos\theta
   \quad\text{ for }\quad \theta = \frac{1}{2}\arcsin\Bigl(\frac{4\pi}{\pi^2+4}\Bigr)\,,
  \]
  where the latter is due to
  \[
   \cot\Bigl(\arcsin\Bigl(\frac{4\pi}{\pi^2+4}\Bigr)\Bigr) = \frac{\sqrt{1-\frac{16\pi^2}{(\pi^2+4)^2}}}{\frac{4\pi}{\pi^2+4}} =
   \frac{\pi^2-4}{4\pi}\,.
  \]
  This completes the proof.
 \end{proof}%
\end{lemma}

So far, Lemma \ref{lem:paramSubst} has been used only to obtain Lemma \ref{lem:lambda_0}. Its whole strength becomes apparent in
connection with Lemma \ref{lem:paramPermut}. This is demonstrated in the following corollary to Lemma \ref{lem:twoParams}, which
states that in \eqref{eq:critPointsNotEqual} the sequences with $l\in\{0,\dots,n-2\}$  do not need to be considered.

\begin{corollary}\label{cor:critPoints:ln}
 In the case \eqref{eq:critPointsNotEqual} in Lemma \ref{lem:critPoints} one has $l=n-1$.
 
 \begin{proof}
  The case $n=1$ is obvious. For $n\ge 2$ let $\lambda=(\lambda_0,\dots,\lambda_n,0,\dots)\in D_n(\theta)$ with
  \[
   \frac{4}{\pi^2+4}>\lambda_0=\dots=\lambda_l > \frac{2}{\pi^2} > \lambda_{l+1}=\dots=\lambda_n>0
  \]
  for some $l\in\{0,\dots,n-2\}$. In particular, one has $0<\lambda_{n-1}=\lambda_n<\frac{2}{\pi^2}$, which implies that
  $0<\tilde\theta:=M(\lambda_{n-1})+M(\lambda_n)<\arcsin\bigl(\frac{2}{\pi}\bigr)$. Hence, it follows from Lemma
  \ref{lem:twoParams} that
  \[
   \max W\bigl((\lambda_{n-1},\lambda_n,0,\dots)\bigr) < T_1(\tilde\theta)\,.
  \]
  By Lemmas \ref{lem:paramPermut} and \ref{lem:paramSubst} one concludes that
  \[
   \max W(\lambda) = \max W\bigl((\lambda_{n-1},\lambda_n,\lambda_0,\dots,\lambda_{n-2},0,\dots)\bigr)<T_n(\theta)\,.
  \]
  This leaves $l=n-1$ as the only possibility in \eqref{eq:critPointsNotEqual}.
 \end{proof}%
\end{corollary}

We now turn to the computation of $T_2(\theta)$ for $\theta\in\bigl[0,\frac{\pi}{2}\bigr]$.

\begin{lemma}\label{lem:threeParams}
 In the interval $\bigl(0,\frac{\pi}{2}\bigr]$ the equation
 \begin{equation}\label{eq:defVartheta}
  \Bigl(1-\frac{2}{\pi}\sin\vartheta\Bigr)^2 = \biggl(1-\frac{2}{\pi}\sin\Bigl(\frac{2\vartheta}{3}\Bigr)\biggr)^3
 \end{equation}
 has a unique solution $\vartheta\in\bigl(\arcsin\bigl(\frac{2}{\pi}\bigr),\frac{\pi}{2}\bigr)$. Moreover, the quantity
 $T_2(\theta)$ has the representation
 \[
  T_2(\theta) = \begin{cases}
                 T_1(\theta) & \text{ for }\quad 0\le\theta\le\vartheta\,,\\[0.1cm]
                 \dfrac{1}{2}-\dfrac{1}{2}\biggl(1-\dfrac{2}{\pi}\sin\Bigl(\dfrac{2\theta}{3}\Bigr)\biggr)^3
                  & \text{ for }\quad \vartheta<\theta\le\frac{\pi}{2}\,.
                \end{cases}
 \]
 In particular, one has $T_1(\theta)<T_2(\theta)$ if $\theta>\vartheta$, and the strict inequality $\max W(\lambda)<T_2( \theta)$
 holds for $\theta\in\bigl(0,\frac{\pi}{2}\bigr]$ and $\lambda=(\lambda_0,\lambda_1,\lambda_2,0, \dots)\in D_2(\theta)$ with
 $\lambda_0=\lambda_1>\lambda_2>0$.
 
 The mapping $\bigl[0,\frac{\pi}{2}\bigr]\ni\theta\mapsto T_2(\theta)$ is strictly increasing, continuous on
 $\bigl[0,\frac{\pi}{2}\bigr]$, and continuously differentiable on $\bigl(0,\frac{\pi}{2}\bigr)\setminus\{\vartheta\}$.
 
 \begin{proof}
  Since $T_2(0)=T_1(0)=0$, the case $\theta=0$ in the representation for $T_2(\theta)$ is obvious. Let
  $\theta\in\bigl(0,\frac{\pi}{2}\bigr]$ be arbitrary. It follows from Lemmas \ref{lem:paramPermut}, \ref{lem:critPoints}, and
  \ref{lem:lambda_0} and Corollary \ref{cor:critPoints:ln} that there are only two sequences in $D_2(\theta)\setminus D_1(\theta)$
  that need to be considered in order to compute $T_2(\theta)$. One of them is $\mu=(\mu_0,\mu_1,\mu_2,0,\dots)$ with
  $\mu_0=\mu_1=\mu_2=\frac{1}{\pi}\sin\bigl(\frac{2\theta}{3}\bigr)$. For this sequence representation \eqref{eq:limitEquiParam}
  yields
  \begin{equation}\label{eq:threeParamsLimitEqui}
   \max W(\mu) = \frac{1}{2}-\frac{1}{2}\biggl(1-\frac{2}{\pi}\sin\Bigl(\frac{2\theta}{3}\Bigr)\biggr)^3\,.
  \end{equation}
  The other sequence in $D_2(\theta)\setminus D_1(\theta)$ that needs to be considered is
  $\lambda=(\lambda_0,\lambda_1,\lambda_2,0,\dots)$, where $\frac{4}{\pi^2+4}>\lambda_0=\lambda_1>\frac{2}{\pi^2}>\lambda_2>0$ and
  $\lambda_0$ and $\lambda_2$ are given by \eqref{eq:critPointsRels} and \eqref{eq:critPointsAlpha}. Using representation
  \eqref{eq:seqExplRepr}, one obtains 
  \begin{equation}\label{eq:threeParamsLimitNotEqual}
   \max W(\lambda) = \frac{1}{2}-\frac{1}{2}(1-2\lambda_0)^2(1-2\lambda_2)\,.
  \end{equation}
  According to Lemma \ref{app:lem3}, this sequence $\lambda$ can exist only if $\theta$ satisfies the two-sided estimate
  $\frac{3}{2}\arcsin\bigl(\frac{2}{\pi}\bigr)<\theta\le\arcsin\bigl(\frac{12+\pi^2}{8\pi}\bigr)+
  \frac{1}{2}\arcsin\bigl(\frac{12-\pi^2}{4\pi}\bigr)$.
  However, if $\lambda$ exists, combining Lemma \ref{app:lem3} with equations \eqref{eq:threeParamsLimitEqui} and
  \eqref{eq:threeParamsLimitNotEqual} yields
  \[
   \max W(\lambda) < \max W(\mu)\,.
  \]
  Therefore, in order to compute $T_2(\theta)$ for $\theta\in\bigl(0,\frac{\pi}{2}\bigr]$, it remains to compare
  \eqref{eq:threeParamsLimitEqui} with $T_1(\theta)$. In particular, for every sequence
  $\lambda=(\lambda_0,\lambda_1,\lambda_2,0,\dots)\in D_2(\theta)$ with $\lambda_0=\lambda_1>\lambda_2>0$ the strict inequality
  $\max W(\lambda) < T_2(\theta)$ holds.
  
  According to Lemma \ref{app:lem2}, there is a unique $\vartheta\in\bigl(\arcsin\bigl(\frac{2}{\pi}\bigr),\frac{\pi}{2}\bigr)$
  such that
  \[
   \Bigl(1-\frac{2}{\pi}\sin\theta\Bigr)^2 < \biggl(1-\frac{2}{\pi}\sin\Bigl(\frac{2\theta}{3}\Bigr)\biggr)^3
   \quad\text{ for }\quad 0 < \theta < \vartheta
  \]
  and
  \[
   \Bigl(1-\frac{2}{\pi}\sin\theta\Bigr)^2 > \biggl(1-\frac{2}{\pi}\sin\Bigl(\frac{2\theta}{3}\Bigr)\biggr)^3
   \quad\text{ for }\quad \vartheta < \theta \le \frac{\pi}{2}\,.
  \]
  These inequalities imply that $\vartheta$ is the unique solution to equation \eqref{eq:defVartheta} in the interval
  $\bigl(0,\frac{\pi}{2}\bigr]$. Moreover, taking into account Lemma \ref{lem:twoParams}, equation \eqref{eq:threeParamsLimitEqui},
  and the inequality $\vartheta>\arcsin\bigl(\frac{2}{\pi}\bigr)$, it follows that $T_1(\theta) < \max W(\mu)$ if and only if
  $\theta>\vartheta$. This proves the claimed representation for $T_2(\theta)$.
  
  By Lemma \ref{lem:twoParams} and the choice of $\vartheta$ it is obvious that the mapping
  $\bigl[0,\frac{\pi}{2}\bigr]\ni\theta\mapsto T_2(\theta)$ is strictly increasing, continuous on $\bigl[0,\frac{\pi}{2}\bigr]$,
  and continuously differentiable on $\bigl(0,\frac{\pi}{2}\bigr)\setminus\{\vartheta\}$.
 \end{proof}%
\end{lemma}

In order to prove Theorem \ref{thm:solOptProb}, it remains to show that $T(\theta)$ coincides with $T_2(\theta)$.

\begin{proposition}\label{prop:solOptProb}
 For every $\theta\in\bigl[0,\frac{\pi}{2}\bigr]$ and $n\ge 2$ one has $T(\theta)=T_n(\theta)=T_2(\theta)$.
 
 \begin{proof}
  Since $T(0)=0$, the case $\theta=0$ is obvious. Let $\theta\in\bigl(0,\frac{\pi}{2}\bigr]$ be arbitrary. As a result of equation
  \eqref{eq:supOptn}, it suffices to show that $T_n(\theta)=T_2(\theta)$ for all $n\ge 3$. Let $n\ge 3$ and let
  $\lambda=(\lambda_j)\in D_n(\theta)\setminus D_{n-1}(\theta)$. The objective is to show that $\max W(\lambda)<T_n(\theta)$.

  First, assume that $\lambda_0=\dots=\lambda_n=\frac{1}{\pi}\sin\bigl(\frac{2\theta}{n+1}\bigr)>0$. We examine the two cases
  $\lambda_0<\frac{2}{\pi^2}$ and $\lambda_0\ge\frac{2}{\pi^2}$. If $\lambda_0<\frac{2}{\pi^2}$, then
  $2M(\lambda_0)<\arcsin\bigl(\frac{2}{\pi}\bigr)$. In this case, it follows from Lemma \ref{lem:twoParams} that
  $\max W\bigl((\lambda_0,\lambda_0,0,\dots)\bigr)<T_1(\tilde\theta)$ with $\tilde\theta=2M(\lambda_0)$. Hence, by Lemma
  \ref{lem:paramSubst} one has $\max W(\lambda)<T_n(\theta)$. If $\lambda_0\ge\frac{2}{\pi^2}$, then
  \[
   (n+1)\arcsin\Bigl(\frac{2}{\pi}\Bigr) \le 2(n+1)M(\lambda_0) = 2\theta \le \pi\,,
  \]
  which is possible only if $n\le3$, that is, $n=3$. In this case, one has
  $\lambda_0=\frac{1}{\pi}\sin\bigl(\frac{\theta}{2}\bigr)$. Taking into account representation \eqref{eq:limitEquiParam}, it
  follows from Lemma \ref{app:lem4} that
  \[
   \max W(\lambda)=\frac{1}{2}-\frac{1}{2}\biggl(1-\frac{2}{\pi}\sin\Bigl(\frac{\theta}{2}\Bigr)\biggr)^4 <
   \frac{1}{2}-\frac{1}{2}\biggl(1-\frac{2}{\pi}\sin\Bigl(\frac{2\theta}{3}\Bigr)\biggr)^3 \le T_2(\theta)\le T_n(\theta)\,.
  \]
  So, one concludes that $\max W(\lambda)<T_n(\theta)$ again.
  
  Now, assume that $\lambda=(\lambda_j)\in D_n(\theta)\setminus D_{n-1}(\theta)$ satisfies
  $\lambda_0=\dots=\lambda_{n-1}>\lambda_n > 0$. Since, in particular, $\lambda_{n-2}=\lambda_{n-1}>\lambda_n>0$, Lemma
  \ref{lem:threeParams} implies that
  \[
   \max W\bigl((\lambda_{n-2},\lambda_{n-1},\lambda_n,0,\dots)\bigr) < T_2(\tilde\theta)\quad\text{ with }\quad
   \tilde\theta = \sum_{j=n-2}^n M(\lambda_j)\,.
  \]
  It follows from Lemmas \ref{lem:paramPermut} and \ref{lem:paramSubst} that
  \[
   \max W(\lambda) = \max W\bigl((\lambda_{n-2},\lambda_{n-1},\lambda_n,\lambda_0,\dots,\lambda_{n-3},0,\dots)\bigr)<T_n(\theta)\,,
  \]
  that is, $\max W(\lambda)< T_n(\theta)$ once again.

  Hence, by Lemmas \ref{lem:paramPermut}, \ref{lem:critPoints}, and \ref{lem:lambda_0} and Corollary \ref{cor:critPoints:ln} the
  inequality $\max W(\lambda) < T_n(\theta)$ holds for all $\lambda\in D_n(\theta)\setminus D_{n-1}(\theta)$, which implies that
  $T_n(\theta)=T_{n-1}(\theta)$. Now the claim follows by induction.
 \end{proof}%
\end{proposition}

We close this section with the following observation, which, together with Remark \ref{rem:estOptimality} above, shows that the
estimate from Theorem \ref{thm:mainResult} is indeed stronger than the previously known estimates.

\begin{remark}\label{rem:AMvsS}
 It follows from the previous considerations that
 \[
  x < T(M_*(x))\quad\text{ for }\quad \frac{4}{\pi^2+4}<x\le c_*\,,
 \]
 where $c_*\in\bigl(0,\frac{1}{2}\bigr)$ and $M_*\colon[0,c_*]\to\bigl[0,\frac{\pi}{2}\bigr]$ are given by \eqref{eq:AMConst} and
 \eqref{eq:AMFunc}, respectively. Indeed, let $x\in\bigl(\frac{4}{\pi^2+4},c_*\bigr]$ be arbitrary and set
 $\theta:=M_*(x)>\frac{1}{2}\arcsin\bigl(\frac{4\pi}{\pi^2+4}\bigr)$. Define $\lambda\in D_2(\theta)$ by
 \[
  \lambda:=\Bigl(\frac{4}{\pi^2+4},\frac{1}{\pi}\sin\Bigl(2\theta-\arcsin\Bigl(\frac{4\pi}{\pi^2+4}\Bigr)\Bigr),0,\dots\Bigr)
  \quad\text{ if }\quad \theta\le\arcsin\Bigl(\frac{4\pi}{\pi^2+4}\Bigr)
 \]
 and by
 \[
  \lambda:=\Bigl(\frac{4}{\pi^2+4},\frac{4}{\pi^2+4},\frac{1}{\pi}\sin\Bigl(2\theta-2\arcsin\Bigl(\frac{4\pi}{\pi^2+4}\Bigr)\Bigr),
  0,\dots\Bigr)\quad\text{ if }\quad \theta>\arcsin\Bigl(\frac{4\pi}{\pi^2+4}\Bigr)\,.
 \]
 Using representation \eqref{eq:seqExplRepr}, a straightforward calculation shows that in both cases one has
 \[
  x=M_*^{-1}(\theta)=\max W(\lambda)\,.
 \]
 
 If $\theta=\arcsin\bigl(\frac{4\pi}{\pi^2+4}\bigr)>\vartheta$ (cf.\ Remark \ref{rem:kappaRepl}), that is,
 $\lambda=\bigl(\frac{4}{\pi^2+4},\frac{4}{\pi^2+4},0,\dots\bigr)$, then it follows from Lemma \ref{lem:threeParams} that
 $\max W(\lambda) \le T_1(\theta) < T_2(\theta)$.
 
 If $\theta\neq\arcsin\bigl(\frac{4\pi}{\pi^2+4}\bigr)$, then the inequality $\max W(\lambda)<T_2(\theta)$ holds since, in this
 case, $\lambda$ is none of the critical points from Lemma \ref{lem:critPoints}.
 
 So, in either case one has $x=\max W(\lambda)<T_2(\theta)=T(\theta)=T(M_*(x))$.
\end{remark}

%%%%%%%%%%%%%%%%%%%%%%%%%%%%%%%%%%%%%%%%%%%%%%%%%%%%%%%%%%%%%%%%%%%%%%%%%%%%%%%%%%%%%%%%%%%%%%%%%%%%%%%%%%%%%%%%%%%%%%%%%%%%%%%%%%%
%%%%%%%%%%%%%%%%%%%%%%%%%%%%%%%%%%%%%%%%%%%%%%%%%%%%%%%%%%%%%%%%%%%%%%%%%%%%%%%%%%%%%%%%%%%%%%%%%%%%%%%%%%%%%%%%%%%%%%%%%%%%%%%%%%%
%%%% Appendix
%%%%%%%%%%%%%%%%%%%%%%%%%%%%%%%%%%%%%%%%%%%%%%%%%%%%%%%%%%%%%%%%%%%%%%%%%%%%%%%%%%%%%%%%%%%%%%%%%%%%%%%%%%%%%%%%%%%%%%%%%%%%%%%%%%%
%%%%%%%%%%%%%%%%%%%%%%%%%%%%%%%%%%%%%%%%%%%%%%%%%%%%%%%%%%%%%%%%%%%%%%%%%%%%%%%%%%%%%%%%%%%%%%%%%%%%%%%%%%%%%%%%%%%%%%%%%%%%%%%%%%%
\begin{appendix}

%%%%%%%%%%%%%%%%%%%%%%%%%%%%%%%%%%%%%%%%%%%%%%%%%%%%%%%%%%%%%%%%%%%%%%%%%%%%%%%%%%%%%%%%%%%%%%%%%%%%%%%%%%%%%%%%%%%%%%%%%%%%%%%%%%%
%%%% Appendix A: Inequalities
%%%%%%%%%%%%%%%%%%%%%%%%%%%%%%%%%%%%%%%%%%%%%%%%%%%%%%%%%%%%%%%%%%%%%%%%%%%%%%%%%%%%%%%%%%%%%%%%%%%%%%%%%%%%%%%%%%%%%%%%%%%%%%%%%%%
\section{Proofs of some inequalities}\label{app:sec:inequalities}

\begin{lemma}\label{app1:lem1}
 The following inequalities hold:
 \begin{enumerate}
  \renewcommand{\theenumi}{\alph{enumi}}
  \item $\frac{2}{\pi^2} + \frac{2\pi-4}{\pi^2}\sin^2\theta < \frac{2}{\pi}\left(1-\frac{1}{\pi}\sin\theta\right)\sin\theta$
        \quad\text{ for }\quad $\arcsin\bigl(\frac{1}{\pi-1}\bigr) < \theta < \frac{\pi}{2}$\,,\label{app:lem1:a}
  \item $\frac{1}{\pi}\sin(2\theta) < \frac{2}{\pi^2} + \frac{\pi^2-4}{2\pi^2}\sin^2\theta$\quad\text{ for }\quad
        $\arctan\bigl(\frac{2}{\pi}\bigr)<\theta\le\frac{\pi}{4}$\,,\label{app:lem1:b}
  \item $\frac{2}{\pi}\left(1-\frac{1}{\pi}\sin\theta\right)\sin\theta < \frac{2}{\pi^2} + \frac{\pi^2-4}{2\pi^2}\sin^2\theta$
        \quad\text{ for }\quad $\theta\neq\arcsin\bigl(\frac{2}{\pi}\bigr)$\,,\label{app:lem1:c}
  \item $\frac{2}{\pi}\left(1-\frac{1}{\pi}\sin\theta\right)\sin\theta < \frac{1}{\pi}\sin(2\theta)$\quad\text{ for }\quad
        $0<\theta<2\arctan\bigl(\frac{1}{\pi}\bigr)$\,,\label{app:lem1:d}
  \item $\frac{2}{\pi}\left(1-\frac{1}{\pi}\sin\theta\right)\sin\theta > \frac{1}{\pi}\sin(2\theta)$\quad\text{ for }\quad
        $2\arctan\bigl(\frac{1}{\pi}\bigr) < \theta < \pi$\,.\label{app:lem1:e}
 \end{enumerate}
 
 \begin{proof}
  One has
  \[
   \begin{aligned}
    \frac{2}{\pi}\Bigl(1-\frac{1}{\pi}\sin\theta\Bigr)\sin\theta &- \Bigl(\frac{2}{\pi^2}+\frac{2\pi-4}{\pi^2}\sin^2\theta\Bigr)\\
    &= - \frac{2(\pi-1)}{\pi^2} \Bigl(\sin^2\theta - \frac{\pi}{\pi-1}\sin\theta+\frac{1}{\pi-1}\Bigr)\\
    &= - \frac{2(\pi-1)}{\pi^2} \biggl(\Bigl(\sin\theta-\frac{\pi}{2(\pi-1)}\Bigr)^2-\frac{(\pi-2)^2}{4(\pi-1)^2}\biggr),
   \end{aligned}
  \]
  which is strictly positive if and only if
  \[
   \Bigl(\sin\theta-\frac{\pi}{2(\pi-1)}\Bigr)^2 < \frac{(\pi-2)^2}{4(\pi-1)^2}\,.
  \]
  A straightforward analysis shows that the last inequality holds for $\arcsin\bigl(\frac{1}{\pi-1}\bigr)<\theta<\frac{\pi}{2}$,
  which proves (\ref{app:lem1:a}).

  For $\theta_0:=\arctan\bigl(\frac{2}{\pi}\bigr)=\frac{1}{2} \arcsin\bigl(\frac{4\pi}{\pi^2+4}\bigr)$ one has
  $\sin(2\theta_0)=\frac{4\pi}{\pi^2+4}$ and $\sin^2\theta_0=\frac{4}{\pi^2+4}$. Thus, the inequality in (b) becomes an equality
  for $\theta=\theta_0$. Therefore, in order to show (b), it suffices to show that the corresponding estimate holds for the
  derivatives of both sides of the inequality, that is,
  \[
   \frac{2}{\pi}\cos(2\theta) < \frac{\pi^2-4}{2\pi^2}\sin(2\theta)\quad\text{ for }\quad \theta_0 < \theta < \frac{\pi}{4}\,.
  \]
  This inequality is equivalent to $\tan(2\theta)>\frac{4\pi}{\pi^2-4}$ for $\theta_0<\theta<\frac{\pi}{4}$, which, in turn,
  follows from $\tan(2\theta_0)=\frac{2\tan\theta_0}{1-\tan^2\theta_0}=\frac{4\pi}{\pi^2-4}$. This implies (\ref{app:lem1:b}).
  
  The claim (\ref{app:lem1:c}) follows immediately from
  \[
   \frac{2}{\pi^2} + \frac{\pi^2-4}{2\pi^2}\sin^2\theta - \frac{2}{\pi}\Bigl(1-\frac{1}{\pi}\sin\theta\Bigr)\sin\theta =
   \frac{1}{2}\Bigl(\frac{2}{\pi}-\sin\theta\Bigr)^2\,.
  \]
  
  Finally, observe that
  \begin{equation}\label{app:eq:lem1:de}
   \frac{1}{\pi}\sin(2\theta) - \frac{2}{\pi}\Bigl(1-\frac{1}{\pi}\sin\theta\Bigr)\sin\theta=
   \frac{2}{\pi}\Bigl(\cos\theta-1+\frac{1}{\pi}\sin\theta\Bigr)\sin\theta\,.
  \end{equation}
  For $0<\theta<\pi$, the right-hand side of \eqref{app:eq:lem1:de} is positive if and only if
  $\frac{1-\cos\theta}{\sin\theta}=\tan\bigl(\frac{\theta}{2}\bigr)$ is less than $\frac{1}{\pi}$. This is the case if and only if
  $\theta < 2\arctan\bigl(\frac{1}{\pi}\bigr)$, which proves (\ref{app:lem1:d}). The proof of claim (\ref{app:lem1:e}) is
  analogous.
 \end{proof}%
\end{lemma}

\begin{lemma}\label{app:lem2}
 There is a unique $\vartheta\in\bigl(\arcsin\bigl(\frac{2}{\pi}\bigr),\frac{\pi}{2}\bigr)$ such that
 \[
  \Bigl(1-\frac{2}{\pi}\sin\theta\Bigr)^2 < \biggl(1-\frac{2}{\pi}\sin\Bigl(\frac{2\theta}{3}\Bigr)\biggr)^3\quad\text{ for }\quad
  0 < \theta < \vartheta
 \]
 and
 \[
  \Bigl(1-\frac{2}{\pi}\sin\theta\Bigr)^2 > \biggl(1-\frac{2}{\pi}\sin\Bigl(\frac{2\theta}{3}\Bigr)\biggr)^3\quad\text{ for }\quad
  \vartheta < \theta \le\frac{\pi}{2}\,.
 \]
 
 \begin{proof}
  Define $u,v,w\colon\R\to\R$ by
  \[
   u(\theta):=\sin\Bigl(\frac{2\theta}{3}\Bigr),\quad
   v(\theta):=\frac{\pi}{2}-\frac{\pi}{2}\Bigl(1-\frac{2}{\pi}\sin\theta\Bigr)^{2/3}\,,\quad\text{ and }\quad
   w(\theta):=u(\theta) - v(\theta)\,.
  \]
  Obviously, the claim is equivalent to the existence of $\vartheta\in\bigl(\arcsin\bigl(\frac{2}{\pi}\bigr),\frac{\pi}{2}\bigr)$
  such that $w(\theta)<0$ for $0<\theta<\vartheta$ and $w(\theta)>0$ for $\vartheta<\theta\le\frac{\pi}{2}$.
  
  Observe that $u'''(\theta)=-\frac{8}{27}\cos\bigl(\frac{2\theta}{3}\bigr)<0$ for $0\le\theta\le\frac{\pi}{2}$. In particular,
  $u''$ is strictly decreasing on the interval $\bigl[0,\frac{\pi}{2}\bigr]$. Moreover, $u'''$ is strictly increasing on
  $\bigl[0,\frac{\pi}{2}\bigr]$, so that the inequality $u'''\ge u'''(0)=-\frac{8}{27}>-\frac{1}{2}$ holds on
  $\bigl[0,\frac{\pi}{2}\bigr]$.
  
  One computes
  \begin{equation}\label{app:eq:lem2:d4v}
   v^{(4)}(\theta) = \frac{2\pi^{1/3}}{81}\,\frac{p(\sin\theta)}{(\pi-2\sin\theta)^{10/3}}\quad\text{ for }\quad
   0\le\theta\le\frac{\pi}{2}\,,
  \end{equation}
  where
  \[
   p(x)=224-72\pi^2 + 27\pi^3x - (160+36\pi^2)x^2 + 108\pi x^3 - 64x^4\,.
  \]
  The polynomial $p$ is strictly increasing on $[0,1]$ and has exactly one root in the interval $(0,1)$. Combining this with
  equation \eqref{app:eq:lem2:d4v}, one obtains that $v^{(4)}$ has a unique zero in $\bigl(0,\frac{\pi}{2}\bigr)$ and that
  $v^{(4)}$ changes its sign from minus to plus there. Observing that $v'''(0)<-\frac{1}{2}$ and $v'''\bigl(\frac{\pi}{2}\bigr)=0$,
  this yields $v'''< 0$ on $\bigl[0,\frac{\pi}{2}\bigr)$, that is, $v''$ is strictly decreasing on $\bigl[0,\frac{\pi}{2}\bigr]$.
  Moreover, it is easy to verify that $v'''\bigl(\frac{\pi}{3}\bigr)< v'''(0)$, so that $v'''\le v'''(0)<-\frac{1}{2}$ on
  $\bigl[0,\frac{\pi}{3}\bigr]$. Since $u'''>-\frac{1}{2}$ on $\bigl[0,\frac{\pi}{2}\bigr]$ as stated above, it follows that
  $w'''=u'''-v'''>0$ on $\bigl[0,\frac{\pi}{3}\bigr]$, that is, $w''$ is strictly increasing on $\bigl[0,\frac{\pi}{3}\bigr]$.
  
  Recall that $u''$ and $v''$ are both decreasing functions on $\bigl[0,\frac{\pi}{2}\bigr]$. Observing the inequality
  $u''\bigl(\frac{\pi}{2}\bigr) > v''\bigl(\frac{\pi}{3}\bigr)$, one deduces that
  \begin{equation}\label{app:eq:lem2:ddw}
   w''(\theta) = u''(\theta) - v''(\theta) \ge u''\Bigl(\frac{\pi}{2}\Bigr) - v''\Bigl(\frac{\pi}{3}\Bigr)>0\quad\text{ for }\quad
   \theta\in\Bigl[\frac{\pi}{3},\frac{\pi}{2}\Bigr]\,.
  \end{equation}
  Moreover, one has $w''(0)<0$. Combining this with \eqref{app:eq:lem2:ddw} and the fact that $w''$ is strictly increasing on
  $\bigl[0,\frac{\pi}{3}\bigr]$, one concludes that $w''$ has a unique zero in the interval $\bigl(0,\frac{\pi}{2}\bigr)$ and that
  $w''$ changes its sign from minus to plus there. Since $w'(0)=0$ and $w'\bigl(\frac{\pi}{2}\bigr)=\frac{1}{3}>0$, it follows that
  $w'$ has a unique zero in $\bigl(0,\frac{\pi}{2}\bigr)$, where it changes its sign from minus to plus. Finally, observing that
  $w(0)=0$ and $w\bigl(\frac{\pi}{2}\bigr)>0$, in the same way one arrives at the conclusion that $w$ has a unique zero
  $\vartheta\in\bigl(0,\frac{\pi}{2}\bigr)$ such that $w(\theta)<0$ for $0<\theta<\vartheta$ and $w(\theta)>0$ for
  $\vartheta<\theta<\frac{\pi}{2}$. As a result of $w\bigl(\arcsin\bigl(\frac{2}{\pi}\bigr)\bigr)<0$, one has
  $\vartheta>\arcsin\bigl(\frac{2}{\pi}\bigr)$.
 \end{proof}%
\end{lemma}

\begin{lemma}\label{app:lem3}
 For $x\in\bigl(\frac{2}{\pi^2},\frac{4}{\pi^2+4}\bigr)$ let
 \begin{equation}\label{app:eq:lem3:alphay}
  \alpha := \frac{\sqrt{1-\pi^2x^2}}{1-2x}\quad\text{ and }\quad
  y:=\frac{4\alpha^2}{\pi^2+4\alpha^2} - x \,.
 \end{equation}
 Then, $\theta:=\arcsin(\pi x)+\frac{1}{2}\arcsin(\pi y)$ satisfies the inequalities
 \begin{equation}\label{app:eq:lem3:theta}
  \frac{3}{2}\arcsin\Bigl(\frac{2}{\pi}\Bigr) < \theta \le
  \arcsin\Bigl(\frac{12+\pi^2}{8\pi}\Bigr)+\frac{1}{2}\arcsin\Bigl(\frac{12-\pi^2}{4\pi}\Bigr)
 \end{equation}
 and
 \begin{equation}\label{app:eq:lem3:estimate}
  \left(1-\frac{2}{\pi}\sin\Bigl(\frac{2\theta}{3}\Bigr)\right)^3 < (1-2x)^2(1-2y)\,.
 \end{equation}
 
 \begin{proof}
  One has $1<\alpha<m:=\frac{\pi}{2}\tan\bigl(\arcsin\bigl(\frac{2}{\pi}\bigr)\bigr)$, $y\in\bigl(0,\frac{2}{\pi^2}\bigr)$, and
  $\alpha=\frac{\sqrt{1-\pi^2y^2}}{1-2y}$ (cf.\ Lemma \ref{lem:critPoints}). Moreover, taking into account that
  $\alpha^2=\frac{1-\pi^2x^2}{(1-2x)^2}$ by \eqref{app:eq:lem3:alphay}, one computes
  \begin{equation}\label{app:eq:lem3:y}
   y=\frac{4-(\pi^2+4)x}{\pi^2+4-4\pi^2x}\,.
  \end{equation}
  Observe that $\alpha\to m$ and $y\to\frac{2}{\pi^2}$ as $x\to\frac{2}{\pi^2}$, and that $\alpha\to1$ and $y\to 0$ as
  $x\to\frac{4}{\pi^2+4}$. With this and taking into account \eqref{app:eq:lem3:y}, it is convenient to consider
  $\alpha=\alpha(x)$, $y=y(x)$, and $\theta=\theta(x)$ as continuous functions of the variable
  $x\in\bigl[\frac{2}{\pi^2},\frac{4}{\pi^2+4}\bigr]$.
  
  Straightforward calculations show that
  \[
   1-2y(x) = \frac{\pi^2-4}{\pi^2+4-4\pi^2x}\cdot (1-2x)\quad\text{ for }\quad \frac{2}{\pi^2}\le x\le \frac{4}{\pi^2+4}\,,
  \]
  so that
  \[
   y'(x)=-\frac{(\pi^2-4)^2}{(\pi^2+4-4\pi^2x)^2} = -\frac{(1-2y(x))^2}{(1-2x)^2}\quad\text{ for }\quad
   \frac{2}{\pi^2}< x< \frac{4}{\pi^2+4}\,.
  \]
  Taking into account that $\alpha(x)=\alpha\bigl(y(x)\bigr)$, that is,
  $\frac{\sqrt{1-\pi^2x^2}}{\sqrt{1-\pi^2y(x)^2}}=\frac{1-2x}{1-2y(x)}$, this leads to
  \begin{equation}\label{app:eq:lem3:dtheta}
   \begin{aligned}
    \theta'(x)
    &=\frac{\pi}{\sqrt{1-\pi^2x^2}} + \frac{\pi y'(x)}{2\sqrt{1-\pi^2y(x)^2}} =
      \frac{\pi}{2\sqrt{1-\pi^2x^2}}\left(2+\frac{1-2x}{1-2y(x)}\cdot y'(x)\right)\\
    &=\frac{\pi}{2\sqrt{1-\pi^2x^2}} \left(2-\frac{\pi^2-4}{\pi^2+4-4\pi^2x}\right) =
      \frac{\pi}{2\sqrt{1-\pi^2x^2}}\cdot\frac{12+\pi^2-8\pi^2x}{\pi^2+4-4\pi^2x}\,.
   \end{aligned}
  \end{equation}
  In particular, $x=\frac{12+\pi^2}{8\pi^2}$ is the only critical point of $\theta$ in the interval
  $\bigl(\frac{2}{\pi^2},\frac{4}{\pi^2+4}\bigr)$ and $\theta'$ changes its sign from plus to minus there. Moreover, using
  $y\bigl(\frac{2}{\pi^2}\bigr)=\frac{2}{\pi^2}$ and $y\bigl(\frac{4}{\pi^2+4}\bigr)=0$, one has
  $\theta\bigl(\frac{2}{\pi^2}\bigr)=\frac{3}{2}\arcsin\bigl(\frac{2}{\pi}\bigr)<\arcsin\bigl(\frac{4\pi}{\pi^2+4}\bigr)=
  \theta\bigl(\frac{4}{\pi^2+4}\bigr)$, so that
  \[
   \frac{3}{2}\arcsin\Bigl(\frac{2}{\pi}\Bigr) < \theta(x) \le \theta\Bigl(\frac{12+\pi^2}{8\pi^2}\Bigr)\quad\text{ for }\quad
   \frac{2}{\pi^2} < x < \frac{4}{\pi^2+4}\,.
  \]
  Since $y\bigl(\frac{12+\pi^2}{8\pi^2}\bigr)=\frac{12-\pi^2}{4\pi^2}$, this proves the two-sided inequality
  \eqref{app:eq:lem3:theta}.
  
  Further calculations show that
  \begin{equation}\label{app:eq:lem3:d2theta}
   \theta''(x) = \frac{\pi^3}{2}\,\frac{p(x)}{(1-\pi^2x^2)^{3/2}\left(\pi^2+4-4\pi^2x\right)^2}\quad\text{ for }\quad
   \frac{2}{\pi^2} < x < \frac{4}{\pi^2+4}\,,
  \end{equation}
  where
  \[
   p(x)=16-4\pi^2 + (48+16\pi^2+\pi^4)x - 8\pi^2(12+\pi^2)x^2 + 32\pi^4x^3\,.
  \]
  The polynomial $p$ is strictly negative on the interval $\bigl[\frac{2}{\pi^2},\frac{4}{\pi^2+4}\bigr]$, so that $\theta'$ is
  strictly decreasing.
  
  Define $w\colon\bigl[\frac{2}{\pi^2},\frac{4}{\pi^2+4}\bigr]\to\R$ by
  \[
   w(x):=(1-2x)^2\cdot \bigl(1-2y(x)\bigr) - \biggl(1-\frac{2}{\pi}\sin\Bigl(\frac{2\theta(x)}{3}\Bigr)\biggr)^3\,.
  \]
  The claim \eqref{app:eq:lem3:estimate} is equivalent to the inequality $w(x)>0$ for $\frac{2}{\pi^2}<x<\frac{4}{\pi^2+4}$. Since
  $y\bigl(\frac{2}{\pi^2}\bigr)=\frac{2}{\pi^2}$ and, hence,
  $\theta\bigl(\frac{2}{\pi^2}\bigr)=\frac{3}{2}\arcsin\bigl(\frac{2}{\pi}\bigr)$, one has $w\bigl(\frac{2}{\pi^2}\bigr)=0$.
  Moreover, a numerical evaluation gives $w\bigl(\frac{4}{\pi^2+4}\bigr)>0$. Therefore, in order to prove $w(x)>0$ for
  $\frac{2}{\pi^2}<x<\frac{4}{\pi^2+4}$, it suffices to show that $w$ has exactly one critical point in the interval
  $\bigl(\frac{2}{\pi^2}, \frac{4}{\pi^2+4}\bigr)$ and that $w$ takes its maximum there.
  
  Using \eqref{app:eq:lem3:dtheta} and taking into account that $\sqrt{1-\pi^2x^2}=\alpha(x)(1-2x)$, one computes
  \[
   \begin{aligned}
    \frac{\dd}{\dd x} (1-2x)^2\bigl(1-2y(x)\bigr)
    &= -4(1-2x)\bigl(1-2y(x)\bigr) - 2(1-2x)^2y'(x)\\
    &= -2(1-2x)\bigl(1-2y(x)\bigr)\left(2+\frac{1-2x}{1-2y(x)}\cdot y'(x)\right)\\
    &= -\frac{4}{\pi}(1-2x)^2\bigl(1-2y(x)\bigr)\alpha(x)\theta'(x)\,.
   \end{aligned}
  \]
  Hence, for $\frac{2}{\pi^2}<x<\frac{4}{\pi^2+4}$ one obtains
  \[
   \begin{aligned}
    w'(x)
    &=-\frac{4}{\pi}\theta'(x)\cdot \left(\alpha(x) (1-2x)^2\bigl(1-2y(x)\bigr)-\left(1-\frac{2}{\pi}\sin\Bigl(\frac{2\theta(x)}
      {3}\Bigr)\right)^2\cos\Bigl(\frac{2\theta(x)}{3}\Bigr)\right)\\
    &=-\frac{4}{\pi}\theta'(x)\cdot \bigl(u(x)-v(x)\bigr)\,,
   \end{aligned}
  \]
  where $u,v\colon\bigl[\frac{2}{\pi^2},\frac{4}{\pi^2+4}\bigr]\to\R$ are given by
  \[
   u(x):=\alpha(x) (1-2x)^2\bigl(1-2y(x)\bigr)\,,\quad
   v(x):=\left(1-\frac{2}{\pi}\sin\Bigl(\frac{2\theta(x)}{3}\Bigr)\right)^2\cos\Bigl(\frac{2\theta(x)}{3}\Bigr)\,.
  \]

  Suppose that the difference $u(x)-v(x)$ is strictly negative for all $x\in\bigl(\frac{2}{\pi^2},\frac{4}{\pi^2+4}\bigr)$. In this
  case, $w'$ and $\theta'$ have the same zeros on $\bigl(\frac{2}{\pi^2},\frac{4}{\pi^2+4}\bigr)$, and $w'(x)$ and $\theta'(x)$
  have the same sign for all $x\in\bigl(\frac{2}{\pi^2},\frac{4}{\pi^2+4}\bigr)$. Combining this with \eqref{app:eq:lem3:dtheta},
  one concludes that $x=\frac{12+\pi^2}{8\pi^2}$ is the only critical point of $w$ in the interval
  $\bigl(\frac{2}{\pi^2},\frac{4}{\pi^2+4}\bigr)$ and that $w$ takes its maximum in this point.
  
  Hence, it remains to show that the difference $u-v$ is indeed strictly negative on
  $\bigl(\frac{2}{\pi^2},\frac{4}{\pi^2+4}\bigr)$. Since
  $\alpha\bigl(\frac{2}{\pi^2}\bigr)=\frac{\pi}{2}\tan\bigl(\arcsin\bigl(\frac{2}{\pi}\bigr)\bigr)$,
  $y\bigl(\frac{2}{\pi^2}\bigr)=\frac{2}{\pi^2}$, and
  $\theta\bigl(\frac{2}{\pi^2}\bigr)=\frac{3}{2}\arcsin\bigl(\frac{2}{\pi}\bigr)$, it is easy to verify that
  $u\bigl(\frac{2}{\pi^2}\bigr)=v\bigl(\frac{2}{\pi^2}\bigr)$ and $u'\bigl(\frac{2}{\pi^2}\bigr)=v'\bigl(\frac{2}{\pi^2}\bigr)<0$.
  Therefore, it suffices to show that $u'<v'$ holds on the whole interval $\bigl(\frac{2}{\pi^2},\frac{4}{\pi^2+4}\bigr)$.
  
  One computes
  \begin{equation}\label{app:eq:lem3:d2u}
   u''(x) = \frac{(\pi^2-4)q(x)}{(1-\pi^2x^2)^{3/2}(\pi^2+4-4\pi^2x)^3}
  \end{equation}
  where
  \[
   \begin{aligned}
    q(x) &= (128-80\pi^2-\pi^6)+12\pi^2(\pi^2+4)^2x-12\pi^2(7\pi^4+24\pi^2+48)x^2\\
         &\quad+32\pi^4(5\pi^2+12)x^3+24\pi^4(\pi^4+16)x^4 -96\pi^6(\pi^2+4)x^5 + 128\pi^8x^6\,.
   \end{aligned}
  \]
  A further analysis shows that $q''$, which is a polynomial of degree $4$, has exactly one root in the interval
  $\bigl[\frac{2}{\pi^2},\frac{4}{\pi^2+4}\bigr]$ and that $q''$ changes its sign from minus to plus there. Moreover, $q'$ takes a
  positive value in this root of $q''$, so that $q'>0$ on $\bigl[\frac{2}{\pi^2},\frac{4}{\pi^2+4}\bigr]$, that is, $q$ is strictly
  increasing on this interval. Since $q\bigl(\frac{4}{\pi^2+4}\bigr)<0$, one concludes that $q<0$ on
  $\bigl[\frac{2}{\pi^2},\frac{4}{\pi^2+4}\bigr]$. It follows from \eqref{app:eq:lem3:d2u} that $u''<0$ on
  $\bigl(\frac{2}{\pi^2},\frac{4}{\pi^2+4}\bigr)$, so that $u'$ is strictly decreasing. In particular, one has
  $u'<u'\bigl(\frac{2}{\pi^2}\bigr)<0$ on $\bigl(\frac{2}{\pi^2}, \frac{4}{\pi^2+4}\bigr)$.
  
  A straightforward calculation yields
  \begin{equation}\label{app:eq:lem3:dv}
   v'(x) = -\frac{2}{3}\biggl(1-\frac{2}{\pi}\sin\Bigl(\frac{2\theta(x)}{3}\Bigr)\biggr)\cdot \theta'(x)\cdot r\biggl(\sin\Bigl(
   \frac{2\theta(x)}{3}\Bigr)\biggr)\,,
  \end{equation}
  where $r(t)=\frac{4}{\pi}+t-\frac{6}{\pi}t^2$. The polynomial $r$ is positive and strictly decreasing on the interval
  $\bigl[\frac{1}{2},1]$. Moreover, taking into account \eqref{app:eq:lem3:theta}, one has
  $\frac{1}{2}<\sin\bigl(\frac{2\theta(x)}{3}\bigr)<1$. Combining this with equation \eqref{app:eq:lem3:dv}, one deduces that
  $v'(x)$ has the opposite sign of $\theta'(x)$ for all $\frac{2}{\pi^2}<x<\frac{4}{\pi^2+4}$. In particular, by
  \eqref{app:eq:lem3:dtheta} it follows that $v'(x)\ge0$ if $x\ge\frac{12+\pi^2}{8\pi^2}$. Since $u'<0$ on
  $\bigl(\frac{2}{\pi^2},\frac{4}{\pi^2+4}\bigr)$, this implies that $v'(x)>u'(x)$ for
  $\frac{12+\pi^2}{8\pi^2}\le x<\frac{4}{\pi^2+4}$. If $\frac{2}{\pi^2}<x<\frac{12+\pi^2}{8\pi^2}$, then one has $\theta'(x)>0$. In
  particular, $\theta$ is strictly increasing on $\bigl(\frac{2}{\pi^2},\frac{12+\pi^2}{8\pi^2}\bigr)$. Recall, that $\theta'$ is
  strictly decreasing by \eqref{app:eq:lem3:d2theta}. Combining all this with equation \eqref{app:eq:lem3:dv} again, one deduces
  that on the interval $\bigl(\frac{2}{\pi^2},\frac{12+\pi^2}{8\pi^2}\bigr)$ the function $-v'$ can be expressed as a product of
  three positive, strictly decreasing terms. Hence, on this interval $v'$ is negative and strictly increasing. Recall that
  $u'<u'\bigl(\frac{2}{\pi^2}\bigr)=v'\bigl(\frac{2}{\pi^2}\bigr)$ on $\bigl(\frac{2}{\pi^2},\frac{4}{\pi^2+4}\bigr)$, which now
  implies that
  \[
   u'(x) < u'\Bigl(\frac{2}{\pi^2}\Bigr) = v'\Bigl(\frac{2}{\pi^2}\Bigr) < v'(x)\quad\text{ for }\quad
   \frac{2}{\pi^2} < x < \frac{12+\pi^2}{8\pi^2}\,.
  \]
  Since the inequality $u'(x)<v'(x)$ has already been shown for $x\ge\frac{12+\pi^2}{8\pi^2}$, one concludes that $u'<v'$ holds on
  the whole interval $\bigl(\frac{2}{\pi^2},\frac{4}{\pi^2+4}\bigr)$. This completes the proof.
 \end{proof}%
\end{lemma}

\begin{lemma}\label{app:lem4}
 One has
 \[
  \biggl(1-\frac{2}{\pi}\sin\Bigl(\frac{2\theta}{3}\Bigr)\biggr)^3 < \left(1-\frac{2}{\pi}\sin\Bigl(\frac{\theta}{2}\Bigr)\right)^4
  \quad\text{ for }\quad 0<\theta\le\frac{\pi}{2}\,.
 \]
 
 \begin{proof}
  The proof is similar to the one of Lemma \ref{app:lem2}. Define $u,v,w\colon\R\to\R$ by
  \[
   u(\theta):=\sin\Bigl(\frac{\theta}{2}\Bigr)\,,\quad
   v(\theta):=\frac{\pi}{2}-\frac{\pi}{2}\biggl(1-\frac{2}{\pi}\sin\Bigl(\frac{2\theta}{3}\Bigr)\biggr)^{3/4}\,,
   \quad\text{ and }\quad w(\theta):=u(\theta)-v(\theta)\,.
  \]
  Obviously, the claim is equivalent to the inequality $w(\theta)<0$ for $0<\theta\le\frac{\pi}{2}$.
  
  Observe that $u'''(\theta)=-\frac{1}{8}\cos\bigl(\frac{\theta}{2}\bigr)<0$ for $0\le\theta\le\frac{\pi}{2}$. In particular,
  $u'''$ is strictly increasing on $\bigl[0,\frac{\pi}{2}\bigr]$ and satisfies $u'''\ge u'''(0)=-\frac{1}{8}$.
  
  One computes
  \begin{equation}\label{app:eq:lem4:d4v}
   v^{(4)}(\theta) = \frac{\pi^{1/4}}{54}\,\frac{p\Bigl(\sin\bigl(\frac{2\theta}{3}\bigr)\Bigr)}
     {\Bigl(\pi-2\sin\bigl(\frac{2\theta}{3}\bigr)\Bigr)^{13/4}}\quad\text{ for }\quad
   0\le\theta\le\frac{\pi}{2}\,,
  \end{equation}
  where
  \[
   p(x) = 45-16\pi^2 + 4\pi(1+2\pi^2)x - (34+20\pi^2)x^2 + 44\pi x^3 - 27x^4\,.
  \]
  The polynomial $p$ is strictly increasing on $\bigl[0,\frac{\sqrt{3}}{2}\bigr]$ and has exactly one root in the interval
  $\bigl(0,\frac{\sqrt{3}}{2}\bigr)$. Combining this with equation \eqref{app:eq:lem4:d4v}, one obtains that $v^{(4)}$ has a unique
  zero in the interval $\bigl(0,\frac{\pi}{2}\bigr)$ and that $v^{(4)}$ changes its sign from minus to plus there. Moreover, it is
  easy to verify that $v'''\bigl(\frac{\pi}{2}\bigl)<v'''(0)<-\frac{1}{8}$. Hence, one has $v'''<-\frac{1}{8}$ on
  $\bigl[0,\frac{\pi}{2}\bigr]$. Since $u'''\ge-\frac{1}{8}$ on $\bigl[0,\frac{\pi}{2}\bigr]$ as stated above, this implies that
  $w'''=u'''-v'''>0$ on $\bigl[0,\frac{\pi}{2}\bigr]$, that is, $w''$ is strictly increasing on $\bigl[0,\frac{\pi}{2}\bigr]$.
  
  With $w''(0)<0$ and $w''\bigl(\frac{\pi}{2}\bigr)>0$ one deduces that $w''$ has a unique zero in $\bigl(0,\frac{\pi}{2}\bigr)$
  and that $w''$ changes its sign from minus to plus there. Since $w'(0)=0$ and $w'\bigl(\frac{\pi}{2}\bigr)>0$, it follows that
  $w'$ has a unique zero in $\bigl(0,\frac{\pi}{2}\bigr)$, where it changes its sign from minus to plus. Finally, observing that
  $w(0)=0$ and $w\bigl(\frac{\pi}{2}\bigr)<0$, one concludes that $w(\theta)<0$ for $0<\theta\le\frac{\pi}{2}$.
 \end{proof}%
\end{lemma}

\end{appendix}

%%%%%%%%%%%%%%%%%%%%%%%%%%%%%%%%%%%%%%%%%%%%%%%%%%%%%%%%%%%%%%%%%%%%%%%%%%%%%%%%%%%%%%%%%%%%%%%%%%%%%%%%%%%%%%%%%%%%%%%%%%%%%%%%%%%
%%% Acknowledgements
%%%%%%%%%%%%%%%%%%%%%%%%%%%%%%%%%%%%%%%%%%%%%%%%%%%%%%%%%%%%%%%%%%%%%%%%%%%%%%%%%%%%%%%%%%%%%%%%%%%%%%%%%%%%%%%%%%%%%%%%%%%%%%%%%%%
\section*{Acknowledgements}
The author is indebted to his Ph.D.\ advisor Vadim Kostrykin for introducing him to this field of research and fruitful
discussions. The author would also like to thank Andr\'{e} H\"anel for a helpful conversation.

%%%%%%%%%%%%%%%%%%%%%%%%%%%%%%%%%%%%%%%%%%%%%%%%%%%%%%%%%%%%%%%%%%%%%%%%%%%%%%%%%%%%%%%%%%%%%%%%%%%%%%%%%%%%%%%%%%%%%%%%%%%%%%%%%%%
%%%%%%%%%%%%%%%%%%%%%%%%%%%%%%%%%%%%%%%%%%%%%%%%%%%%%%%%%%%%%%%%%%%%%%%%%%%%%%%%%%%%%%%%%%%%%%%%%%%%%%%%%%%%%%%%%%%%%%%%%%%%%%%%%%%
%%% Bibliography
%%%%%%%%%%%%%%%%%%%%%%%%%%%%%%%%%%%%%%%%%%%%%%%%%%%%%%%%%%%%%%%%%%%%%%%%%%%%%%%%%%%%%%%%%%%%%%%%%%%%%%%%%%%%%%%%%%%%%%%%%%%%%%%%%%%
%%%%%%%%%%%%%%%%%%%%%%%%%%%%%%%%%%%%%%%%%%%%%%%%%%%%%%%%%%%%%%%%%%%%%%%%%%%%%%%%%%%%%%%%%%%%%%%%%%%%%%%%%%%%%%%%%%%%%%%%%%%%%%%%%%%


\begin{thebibliography}{[10]}
 \bibitem{AG93} N.\ I.\ Akhiezer, I.\ M.\ Glazman, \emph{Theory of Linear Operators in Hilbert Space}, Dover Publications, New York
   (1993).
 \bibitem{AM13} S.\ Albeverio, A.\ K.\ Motovilov, \emph{Sharpening the norm bound in the subspace perturbation theory}, Complex
   Anal.\ Oper.\ Theory \textbf{7} (2013), 1389--1416.
 \bibitem{Brown93} L.\ G.\ Brown, \emph{The rectifiable metric on the set of closed subspaces of Hilbert space},
   Trans.\ Amer.\ Math.\ Soc.\ \textbf{337} (1993), 279--289.
 \bibitem{Davis63} C.\ Davis, \emph{The rotation of eigenvectors by a perturbation}, J.\ Math.\ Anal.\ Appl.\ \textbf{6} (1963),
   159--173.
 \bibitem{DK70} C.\ Davis, W.\ M.\ Kahan, \emph{The rotation of eigenvectors by a perturbation. III}, SIAM
   J.\ Numer.\ Anal.\ \textbf{7} (1970), 1--46.
 \bibitem{Kato66} T.\ Kato, \emph{Perturbation Theory for Linear Operators}, Springer-Verlag, Berlin Heidelberg (1966).
 \bibitem{KMM03} V.\ Kostrykin, K.\ A.\ Makarov, A.\ K.\ Motovilov, \emph{On a subspace perturbation problem},
   Proc.\ Amer.\ Math.\ Soc.\ \textbf{131} (2003), 3469--3476.
 \bibitem{KMM03:2} V.\ Kostrykin, K.\ A.\ Makarov, A.\ K.\ Motovilov, \emph{Existence and uniqueness of solutions to the operator
   Riccati equation. A geometric approach}, Contemp.\ Math.\ \textbf{327}, Amer.\ Math.\ Soc.\ (2003), 181--198.
 \bibitem{KMM07} V.\ Kostrykin, K.\ A.\ Makarov, A.\ K.\ Motovilov, \emph{Perturbation of spectra and spectral subspaces},
   Trans.\ Amer.\ Math.\ Soc.\ \textbf{359} (2007), 77--89.
 \bibitem{MS10} K.\ A.\ Makarov, A.\ Seelmann, \emph{Metric properties of the set of orthogonal projections and their applications
   to operator perturbation theory}, e-print arXiv:1007.1575 [math.SP] (2010).
 \bibitem{Seel13} A.\ Seelmann, \emph{Notes on the $\sin2\Theta$ theorem}, e-print arXiv:1310.2036 [math.SP] (2013).
\end{thebibliography}
\end{document}